\documentclass[11pt,english]{amsart}
\usepackage[T1]{fontenc}
\usepackage[utf8]{inputenc}
\usepackage[letterpaper]{geometry}
\usepackage{amsmath,amsfonts,bm,graphicx,amssymb}
\usepackage{color}

\makeatletter

\newtheorem{theorem}{Theorem}[section]
\newtheorem{lemma}[theorem]{Lemma}

\newtheorem{cor}[theorem]{Corollary}

\theoremstyle{definition}

\newtheorem{example}[theorem]{Example}

\theoremstyle{remark}
\newtheorem{remark}[theorem]{Remark}

\numberwithin{equation}{section}

\def\dlim{\stackrel{ d}{=}}

\def\1{\mathbf{1}}
\def\R{\mathbb{R}}

\def\X{\mathbf{X}}
\def\Y{\mathbf{Y}}
\def\E{\mathbb{E}}

\newcommand{\ip}[2]{{\langle #1,#2\rangle}}

\def\eqfd{\stackrel{ fd}{=}}
\def\Sa{\mathcal{S}}

\makeatother

\usepackage{babel}

\begin{document}
\sloppy
\title[Modeling and simulation with operator scaling]{Modeling and simulation with operator scaling}
\footnote{{\it AMS classification (2000)}: 60G51, 60G52, 68U20, 60F05.}

\author{Serge Cohen}
\address{Serge Cohen, Universit\'{e}
de Toulouse, Universit{\'e} Paul Sabatier, Institut de Math\'{e}matiques
de Toulouse.  F-31062 Toulouse.
France. }
\email{Serge.Cohen@math.univ-toulouse.fr}
\urladdr{http://www.math.univ-toulouse.fr/$\sim$cohen/}

\author{Mark M. Meerschaert}
\address{Mark M. Meerschaert, Department of Mathematics, University of Nevada, Reno NV 89557 USA}
\email{mcubed\@@{}stt.msu.edu }
\urladdr{http://www.stt.msu.edu/$\sim$mcubed}

\author{Jan Rosi\'nski}
\address{Jan Rosi\'nski, Department of Mathematics, University of Tennessee, Knoxville, TN
37996, USA.}
\email{rosinski@math.utk.edu}
\urladdr{http://www.math.utk.edu/$\sim$rosinski/}

\date{December 23, 2009}

\begin{abstract}
Self-similar processes are useful in modeling diverse phenomena that exhibit scaling properties.  Operator scaling allows a different scale factor in each coordinate.  This paper develops practical methods for modeling and simulating stochastic processes with operator scaling.  A simulation method for operator stable L\'evy processes is developed, based on a series representation, along with a Gaussian approximation of the small jumps.  Several examples are given to illustrate practical applications.  A classification of operator stable L\'evy processes in two dimensions is provided according to their exponents and symmetry groups.  We conclude with some remarks and extensions to general operator self-similar processes.
\end{abstract}



\thanks{M. M. Meerschaert was partially
supported by NSF grants DMS-0803360 and EAR-0823965.  J. Rosi\'nski was partially supported by NSA grant MSPF-50G-049.  S. Cohen wishes to thank the Department of Mathematics, University of Tennessee, for their hospitality during a visit that initiated this project.}

\maketitle

\section{Introduction}\label{sec:oss}

Self-similar processes form an important and useful class, favored in practical applications for their nice scaling properties, see for example the recent books of Embrechts and Maejima \cite{em:mae} and Sheluhin et al. \cite{Sheluhin}. In finance, self-similar processes such as fractional Brownian motion and stable L\'evy motion are used to model prices (or log returns).  Recall that a stochastic process $\X=\{X(t)\}_{t\ge0}$ taking values in $\R^{d}$ is \emph{self-similar} if
\begin{equation}\label{eq:ss}
\{X(ct)\}_{t\ge0} \, \eqfd \, \{c^\beta X(t)\}_{t\ge0}
\end{equation}
 at every scale $c>0$. Here $\eqfd$ indicates equality of finite dimensional distributions, and we assume $\X$ is stochastically continuous with $X(0)=0$.  The parameter $\beta>0$ is often called the Hurst index \cite{hurst}.
Operator self-similar processes allow the scaling factor (Hurst index) to vary with the coordinate.  Therefore, a process $\X$ as above is said to be \emph{operator self-similar} \emph{(o.s.s.)}
if there exists a linear operator $B\in\mathrm{GL}(\R^{d})$ such that
\begin{equation}
\left\{ X(ct)\right\} _{t\ge0}\,\eqfd\,\left\{ c^{B}X(t)\right\} _{t\ge0},\label{ss}
\end{equation}
for all $c>0$, where the matrix power $c^{B}:=\exp(B\log c)$. The linear operator $B$ in \eqref{ss}
is called an \emph{exponent} of the operator self-similar process
$\X$. If $B=\beta I$ for some
$\beta>0$, then $\X$ is self-similar.  If $B$ is diagonal, then the marginals of $\X$ are self-similar, and the Hurst index can vary with the coordinate.
This is important in modeling many real world phenomena.   Rachev and Mittnik \cite{MittnikR} show that the scaling index will vary between elements of a portfolio containing different stocks (see also Meerschaert and Scheffler \cite{portfolio}).  In ground water hydrology, Benson et al. \cite{WRRapply,BensonTiPM} use operator self-similar processes to characterize spreading plumes of pollution particles.  Because the structure of the intervening porous medium is not isotropic, the scaling properties vary with direction, see Meerschaert et al. \cite{gADE} and Zhang et al. \cite{RW2d}.  In tick-by-tick analysis of financial data, it is useful to consider the waiting time between trades and the resulting price change as a two dimensional random vector.  Meerschaert and Scalas \cite{coupleEcon} show that different indices apply to price jumps and waiting times, see also Scalas et al. \cite{SGMM}. Results and further references on {o.s.s.}~processes can be found in \cite[Chapter 9]{em:mae} and \cite[Chapter 11]{Meerschaert01}, see also the pioneering work of Hudson and Mason \cite{Hudson82}.

This paper focuses on operator self-similar L\'evy processes. They belong to the class of operator stable processes, which is reviewed in Section \ref{sec:os}. Operator stable processes admit parametrization by their operator exponents and spectral measures, which is a starting point for their analysis.
Section \ref{sec:series} presents a method for simulating the sample paths of operator stable L\'evy processes, consisting of a shot noise representation for the large jumps, and a Gaussian approximation of the small jumps.  Theorem \ref{th:ost} justifies this method and provides a bound on the error resulting from replacing the small jump part by a multidimensional Brownian motion. Section \ref{sec:sim} presents a number of examples to illustrate the method, including some practical applications. The uniqueness of the exponents, which is a critical issue in modeling, is determined by symmetries of such processes (see Sections \ref{sec:os} and \ref{xxx}).    Section \ref{sec:sym-oss} provides a classification of operator stable L\'evy processes in two dimensions according to their exponents and symmetry groups. There we also give a complete, explicit description of the possible symmetries in terms of the exponent and the spectral measure. Using these results, it is possible to construct an operator self-similar process with any given exponent and any admissible symmetry group. Finally, in Section \ref{xxx} we provide some concluding remarks and extensions to general operator self-similar processes.

\bigskip	

\section{Operator stable processes}\label{sec:os}

We say that a L\'evy process $\X=\{X(t)\}_{t\ge0}$ taking values in $\R^d$ is operator stable with exponent $B\in\text{GL}(\mathbb{R}^{d})$ if for every $t>0$ there exists a vector $b(t)\in\mathbb{R}^{d}$
such that
\begin{equation}
X(t)\,\dlim\, t^{B}X(1)+b(t)\label{eq:def-op-stable-mu}
\end{equation}
where $\dlim$ means equal in distribution.  We say that $\X$ is strictly operator stable when $b(t)=0$ for all $t>0$. A L\'evy process is operator self-similar if and only if it is strictly operator stable, in which case the exponents coincide \cite[Theorem 7]{Hudson82}.   In general, if $\X$ is operator stable and 1 is not an eigenvalue of the exponent $B$, then there exists a vector $a$ such that $\{X(t)-at\}_{t\ge0}$ is strictly operator stable; a complete description of strictly operator stable processes is given by Sato \cite{Sato}.   Henceforth we will always assume that the infinitely divisible distribution $\mu=\mathcal{L}(X(1))$ is full dimensional, i.e., not supported on a lower dimensional hyperplane.  The distributional properties of $\mu$ determine those of $\X$.  Indeed, two L\'evy processes $\X$ and $\Y$ have the same finite dimensional distributions if and only if $X(1)$ and $Y(1)$ are identically distributed.

A comprehensive introduction to operator stable laws can be found
in the monographs \cite{Jurek93} and \cite{Meerschaert01}.  Since an operator stable law $\mu$ is infinitely divisible, its characteristic function can be expressed in terms of the L\'evy representation (see, e.g., \cite[Theorem 3.1.11]{Meerschaert01}).  The necessary
and sufficient condition for a $d\times d$ matrix $B$ to be an exponent of a full operator stable law is
that all the roots of the minimal polynomial of $B$ have real parts
greater than or equal to $1/2$, and all the roots with real part equal to $1/2$ are simple, see \cite[Theorem 4.6.12]{Jurek93}.  Furthermore, we can decompose the space $\R^d$ into a direct sum of two $B$-invariant subspaces $\R^d=V_1\oplus V_2$ and write $\mu=\mu_1*\mu_2$, where $\mu_1$ is a normal law supported on $V_1$, and $\mu_2$ is an infinitely divisible law supported on $V_2$ having no normal component.  The restriction of $B$ to $V_1$ has all eigenvalues with real part equal to $1/2$, and the restriction of $B$ to $V_2$ has all eigenvalues with real part strictly greater than $1/2$.

The exponent $B$ in \eqref{ss} need not be unique.  The set of all possible exponents of $\mu$ is given by \cite[Theorem 4.6.7]{Jurek93}:
\begin{equation}\label{exp}
	\mathcal{E}(\mu) = B + T\Sa(\mu)
\end{equation}
 where $B\in\mathcal{E}(\mu)$ is arbitrary.  Here
\begin{equation}\label{sg-mu}
	\Sa(\mu) := \big\{A\in \mathrm{GL}(\R^d): \, A\mu = \mu\ast \delta_x \ \text{for some}  \ x \in \R^d\big\}
\end{equation}
is the symmetry group of the probability measure $\mu$, and $T\mathcal{S}(\mu)$ is the tangent space of $\mathcal{S}(\mu)$ at the identity.  The tangent space consists of all tangent vectors $x'(0)$ where $x(t)$ is a smooth curve on $\mathcal{S}(\mu)$ with $x(0)=I$.  It is a linear space closed under the Lie bracket $[A,B]=AB-BA$.  If $\mathcal{S}(\mu)$ is finite, then $T\mathcal{S}(\mu)=\{0\}$, and this is the only case when the exponent in \eqref{eq:def-op-stable-mu} is unique.  If $A \in \mathcal{S}(\mu)$ and $B$ is an exponent of $\mu$, then so is $A^{-1}BA$.  When the exponent is unique, we must have $AB=BA$, so $B$ commutes with $\mathcal{S}(\mu)$.  The use of commuting exponents simplifies the analysis of $\mathcal{E}(\mu)$. Every operator stable law $\mu$ has an exponent $B_c$ that commutes with $\Sa(\mu)$, see \cite[Theorem 4.7.1]{Jurek93}.  If $\mu$ is operator stable with $\mathcal{S}(\X)=\mathcal{O}_d$,
the orthogonal group on $\R^d$, then $B_c=\beta I$ for some $\beta>0$ is the only commuting exponent, and $\mu$ is multivariable stable with index $1/\beta$.  Since $T(\mathcal{O}_d)=\mathcal{Q}_d$ is the linear space of skew symmetric matrices, we get from \eqref{exp} that
\begin{equation}
\mathcal{E}(\mu)=\beta I+\mathcal{Q}_d. \label{expO}
\end{equation}
Recall that a matrix $Q$ is skew-symmetric if $Q^{\top}=-Q$, where $Q^{\top}$
is the transpose of $Q$.

If $\mathcal{S}(\mu)$ is an arbitrary compact subgroup of $\mathrm{GL}(\R^{d}),$
then by a classical result of algebra (see, e.g., \cite[Theorem 5]{Billingsley1966}) there exists a symmetric positive-definite
matrix $W$ and a compact subgroup $\mathcal{G}$ of the orthogonal
group $\mathcal{O}_d$ such that
\begin{equation}
\mathcal{S}(\mu)=W^{-1}\mathcal{G} W. \label{sg1}
\end{equation}
Then \eqref{exp} becomes
\begin{equation}\label{exp1}
\mathcal{E}(\mu)=B + W^{-1} \mathcal{H} W,
\end{equation}
where $\mathcal{H}$ is the tangent space of $\mathcal{G}$.

Theorem 2 in \cite{Meerschaert95} implies that a compact subgroup $\mathcal{G}$ of $\mathrm{GL}(\R^d)$ can be a symmetry group of a full dimensional probability distribution on $\R^d$ if and only if it is maximal, meaning that $\mathcal{G}$ cannot be strictly contained in any other subgroup that has the same orbits.  For example, the special orthogonal group $\mathcal{O}_d^+$ is not maximal because $\mathcal{O}_d^+ x=\mathcal{O}_d x$ for every $x\in \R^d$, and $\mathcal{O}_d^+$ is a proper subgroup of $\mathcal{O}_d$. Consequently, $\mathcal{O}_d^+$ cannot be the symmetry group of any full dimensional probability measure on $\R^d$.  Actually Theorem 2 in \cite{Meerschaert95} characterizes the strict symmetry group of $\mu$ defined by
\begin{equation}\label{ssg-mu}
{\mathcal S}_0(\mu):= \big\{A\in \mathrm{GL}(\R^d): \, A\mu = \mu\} .
\end{equation}
However, Theorem 5 in Billingsley \cite{Billingsley1966} implies that $\Sa(\mu) = \mathcal{S}_0(\mu\ast \delta_a)$ for some $a\in \R^d$.  Hence $\Sa(\mu)$ must be maximal as well.

In this work we will assume that the operator stable law $\mu$ has no Gaussian component, so that all the roots of the minimal polynomial of $B$ have real parts greater than $1/2$. For a given exponent $B$, consider a norm $\|\cdot\|_B$  on $\R^d$ satisfying the following conditions
\begin{itemize}
	\item[(i)] for each $x\in \R^d$, $x\ne 0$, the map  $t \mapsto \|t^Bx\|_B$ is strictly increasing in $t>0$,
	\item[(ii)]  the map $(t,x) \mapsto t^B x$ from $(0,\infty)\times S_B$ onto $\R^d \setminus \{0\}$  is a homeomorphism,
\end{itemize}
where $S_B=\left\{ x\in\R^{d}:\,\|x\|_{B}=1\right\}$ is the unit sphere with respect to $\|\cdot\|_B$.
There are many ways of constructing such norms. For example, Jurek and Mason \cite[Proposition 4.3.4]{Jurek93} propose
\begin{equation}\label{JMnorm}
\|x\|_B=  \left( \int_0^1 \|s^B x\|^p s^{-1} \, ds \right)^{1/p}
\end{equation}
where $1\le p<\infty$ and $\|\cdot\|$ is {\em any} norm on $\R^d$. Meerschaert and Scheffler \cite[Remark 6.1.6]{Meerschaert01}  observe that if the matrix $B$ is in the Jordan form, then the Euclidean norm satisfies (i)-(ii). Moreover, in this case  the function $t\mapsto \Vert t^{B}x\Vert$ is  regularly varying. Under conditions (i)-(ii) we have the following {\em polar decomposition} of the L\'evy measure of an operator stable law
\begin{equation} \label{eq:Lm-ost}
\nu(E)=\int_{S_{B}}\int_{0}^{\infty}\mathbf{1}_{E}(s^{B}u)s^{-2}\, ds\lambda(du),
 \qquad E \in \mathcal{B}(\mathbb{R}^{d}),
\end{equation}
where $\lambda$ is a finite Borel measure on $S_{B}$ called the {\em spectral measure} of $\mu$. The spectral measure is given by
\begin{equation}\label{eq:sm}
	\lambda(F)=\nu(\{x: x=t^Bu, \ \text{for some } (t,u) \in [1,\infty)\times F \}),  \quad F \in \mathcal{B}(S_B)
\end{equation}
and then it follows from \eqref{eq:Lm-ost} and \eqref{eq:sm} that the spectral measure $\lambda$ is uniquely determined for a given L\'evy measure $\nu$, exponent $B$, and norm $\|x\|_{B}$.   The choice of $\|\cdot\|_B$ is a matter of convenience.
For example, if $B$ is in Jordan form, then the Euclidean norm $\|\cdot\|$ is a natural choice for $\|\cdot\|_B$.  Since $\mu$ is full, the smallest linear space supporting the L\'evy measure $\nu$ is $\R^d$ \cite[Proposition 3.1.20]{Meerschaert01}. Moreover, we have a relation between the symmetries of $\mu$ and the strict symmetries of $\nu$
\begin{equation}\label{sm=sn}
		\Sa(\mu) = \mathcal{S}_0(\nu):=\big\{A\in \mathrm{GL}(\R^d): \, A\nu = \nu \big\} ,
\end{equation}
which is valid for any infinitely divisible distribution without Gaussian part.

\section{Accelerated series representation}\label{sec:series}

Let $\X=\{X(t)\}_{t\ge0}$ be a proper operator stable L\'evy process
with exponent $B$, no Gaussian component, and characteristic function in the L\'evy-Khintchine form
\begin{equation} \label{eq:ch.f.}
\log\mathbb{E}e^{i\langle y,X(1)\rangle}=i\langle y,x_{0}\rangle+\int_{\mathbb{R}^{d}}(e^{i\langle y,x\rangle}-1-i\langle y,x\rangle\mathbf{1}_{\{\Vert x\Vert\le1\}})\,\nu(dx).
\end{equation}
In this section, we present a practical method for simulating sample paths of this process.  Our method is based on a series representation \cite{Rosinski90} in which the small jumps are approximated by a Brownian motion \cite{Cohen07}.  The Gaussian approximation of small jumps accelerates the convergence of the series representation, allowing a fast and accurate simulation of sample paths.
Assume that the L\'evy measure $\nu$ is given by \eqref{eq:Lm-ost},
where $S_B$ is the unit sphere with respect to a norm $\|\cdot\|_B$ satisfying conditions (i)-(ii) of the previous section and $\lambda$ is a finite measure on $S_B$. Our approach to simulation of $\X$ is based on a series
expansion and a Gaussian approximation of the remainder of such series.
Such a series expansion falls into a general category of shot noise
representations and is a consequence of the polar decomposition \eqref{eq:Lm-ost},
see remark following \cite[Corollary 4.4]{Rosinski90}. Namely, for
any fixed $T>0$,

\begin{equation}
X(t)= x_0+\sum_{j=1}^{\infty}\left\{ \mathbf{1}_{(0,\, t]}(\tau_{j})\left(\frac{\Gamma_{j}}{T\lambda(S_{B})}\right)^{-B}v_{j}- \frac{t}{T}c_{j} \right\} ,\quad t\in[0,T],\label{eq:series}
\end{equation}
 where $\{\tau_{j}\}$ is an iid sequence of uniform on $[0, T]$
random variables, $\{\Gamma_{j}\}$ forms a Poisson point process
on $(0,\infty)$ with the Lebesgue intensity measure, $\{v_{j}\}$
is an iid sequence on $S_{B}$ with the common distribution $\lambda/\lambda(S_{B}),$ and
\begin{equation}
c_j=\int_{j-1}^j  \int_{\|x\|\leq 1} x \sigma_r(dx)\,dr ,\label{eq:cj1}
\end{equation}
where
\begin{equation}
\sigma_r(A)=P\left(\left(\frac{r}{T\lambda(S_{B})}\right)^{-B} v_1\in A\right) \label{eq:cj2}
\end{equation}
 (see \cite[Eq.\ (5.6)]{Rosinski01b}).
The random sequences $\{\tau_{j}\}$, $\{\Gamma_{j}\}$, and $\{v_{j}\}$ are
independent. The series \eqref{eq:series} converges pathwise uniformly
on $[0,T]$ with probability one, see \cite[Theorem 5.1]{Rosinski01b}.
Fix $\epsilon\in (0,1]$ and define $\mathbf{N}^{\epsilon}=\left\{ N^{\epsilon}(t)\right\} _{t\in[0,T]}$
by
\begin{equation}
N^{\epsilon}(t)=\sum_{\Gamma_{j}\le T\lambda(S_B)/\epsilon }I_{(0,\, t]}(\tau_{j})\left(\frac{\Gamma_{j}}{T\lambda(S_{B})}\right)^{-B}v_{j}.
\label{eq:ostN}
\end{equation}
It is elementary to check that $\mathbf{N}^{\epsilon}$ is a compound
Poisson process with characteristic function
\[
\mathbf{\mathbb{E}}\exp i\langle y,N^{\epsilon}(t)\rangle=\exp\left\{ t\int_{S_{B}}\int_{\epsilon}^{\infty}(e^{i\langle y,s^{B}u\rangle}-1)s^{-2}\, ds\lambda(du)\right\}\,.
\]
To see this: Observe that the number of terms $M_\epsilon$ in the sum \eqref{eq:ostN} is Poisson with mean $\theta_\epsilon=T\lambda(S_B)/\epsilon$; condition on $M_\epsilon=n$ in the characteristic function, noting that $(\Gamma_1/\theta_{\epsilon},\ldots,\Gamma_n/\theta_{\epsilon})$ is equal in distribution to the vector of order statistics from $n$ IID standard uniform random variables; permute the order statistics;  and rewrite the characteristic function as an integral.
Thus $\mathbf{N}^{\epsilon}$ has the L\'evy measure
\[
\nu^{\epsilon}(A)=\int_{S_{B}}\int_{\epsilon}^{\infty}\mathbf{1}_{A}(s^{B}u)s^{-2}\, ds\lambda(du).
\]
The remainder
\begin{equation}
R_{\epsilon}(t)= X(t) - N^{\epsilon}(t),\label{eq:R}
\end{equation}
is a L\'evy process independent of $\mathbf{N}^{\epsilon}$ and $R_{\epsilon}(1)$
has L\'evy measure $\nu_{\epsilon}$ of bounded support given by
\begin{equation}\label{Rlevy}
\nu_{\epsilon}(A)=\int_{S_{B}}\int_{0}^{\epsilon}\mathbf{1}_{A}(s^{B}u)s^{-2}\, ds\lambda(du).
\end{equation}
Therefore, all moments of $R_{\epsilon}(1)$ are finite. A straightforward computation shows that
\begin{equation}\label{eq:drift}
a_{\epsilon}: = \E R_{\epsilon}(1) =x_{0}+ \int_{\Vert x\Vert>1}x\,\nu_{\epsilon}(dx)-\int_{\Vert x\Vert\le1}x\,\nu^{\epsilon}(dx).
\end{equation}
Then we have
$$
X(t) =t a_{\epsilon} +  N^{\epsilon}(t) + \{R_{\epsilon}(t)-\mathbb{E}[R_{\epsilon}(t)]\}.
$$
In our main theorem we will show that under certain matrix scaling $R_{\epsilon}(t)-\mathbb{E}[R_{\epsilon}(t)]$
converges to a standard Brownian motion in $\R^{d}.$  Hence any operator stable L\'evy process can be faithfully approximated by the sum of two independent component processes, a compound Poisson and a Brownian motion with drift.  To this end
we will use Theorem 3.1 in \cite{Cohen07}.  A simple computation (see \cite[Eq.\ (2.3)]{Cohen07}) shows that the covariance matrix $\Sigma_{\epsilon}$ of $R_{\epsilon}(1)$ is given by
\begin{equation}\begin{split}\label{eq:s1}
\Sigma_{\epsilon} & =\mathbb{E}\left[(R_{\epsilon}(1)-\mathbb{E}[R_{\epsilon}(1)]) (R_{\epsilon}(1)-\mathbb{E}[R_{\epsilon}(1)])^{\top}\right]\\
& =\int_{S_{B}}\int_{0}^{\epsilon}(s^{B}u)(s^{B}u)^{\top}\, s^{-2}\, ds\lambda(du)=\int_{0}^{\epsilon}s^{B}\Lambda(s^{B})^{\top}\, s^{-2}ds,
\end{split}\end{equation}
 where $\Lambda$ is given by
\begin{equation}
\Lambda=\int_{S_{B}}uu^{\top}\,\lambda(du).\label{eq:L}
\end{equation}
We observe the following scaling
 \begin{equation}
\Sigma_{\epsilon}=\epsilon^{-1}\int_{0}^{1}(\epsilon\, r)^{B}\Lambda((\epsilon\, r)^{B})^{\top}\, r^{-2}dr=\epsilon^{-1}\epsilon^{B}\Sigma_{1}(\epsilon^{B})^{\top}.\label{eq:sigma-sc}
\end{equation}
Theorem 3.1 in \cite{Cohen07} assumes that $\Sigma_{\epsilon}$ is
nonsingular for all $\epsilon>0.$ Since we consider the spectral measure together with the exponent $B$ as primary parameters of an operator stable law, it is natural to state the nonsingularity condition in terms of these characteristics.
 Let $\text{lin}_{B}(\text{supp}\lambda)$
denote the smallest $B$-invariant subspace of $\mathbb{R}^{d}$
containing the support of $\lambda$. If $\nu$ is as in \eqref{eq:Lm-ost},
then the support of $\nu$ is not contained in a proper subspace of
$\mathbb{R}^{d}$ if and only if
\begin{equation}
\text{lin}_{B}(\text{supp}\ \lambda)=\mathbb{R}^{d}\label{eq:linB}
\end{equation}
cf. \cite{Jurek93}, Corollary 4.3.5. In particular, \eqref{eq:linB}
holds when $\lambda$ is not concentrated on a proper subspace of $\mathbb{R}^{d}$.

As we have stated in Section \ref{sec:oss}, since $\X$ does not have  Gaussian component,
\begin{equation}
b_{\ast}:=\min\{b_{1},\ldots,b_{d}\}>\frac{1}{2}\,,\label{eq:b}
\end{equation}
 where $b_{1},\dots,b_{d}$ are the real parts of the eigenvalues
of $B$. This quantity does not depend on a choice of $B$ because the real parts of eigenvalues of all exponents of $\X$ are the same (see \cite[Corollary 7.2.12]{Meerschaert01}).

\medskip

\begin{theorem} \label{th:ost} Let $\mathbf{X}$ be an operator
stable L\'evy process with exponent $B$ and let the L\'evy measure of
$X(1)$ be given by \eqref{eq:Lm-ost} such that \eqref{eq:linB}
holds. Fix $T>0$ and let $\mathbf{N}^{\epsilon}$ be as in \eqref{eq:ostN}, $\mathbf{W}$
be a standard Brownian motion in $\mathbb{R}^{d}$ independent of
$\mathbf{N}^{\epsilon}$, and $\mathbf{a}_{\epsilon}=\{a_{\epsilon}t\}_{t\ge0}$
be a drift determined by \eqref{eq:drift}.  Define
\begin{equation}\label{Aepsdef}
A_\epsilon=\epsilon^{-1/2}\epsilon^{B}\Sigma_{1}^{1/2}
\end{equation}
where $\Sigma_{1}$ is given by \eqref{eq:s1} with $\epsilon=1$.

Then, for every $\epsilon\in(0,1]$ there exists a c\'adl\'ag process
$\mathbf{Y}_{\epsilon}$ such that on $[0,T]$
\begin{equation}
\mathbf{X\,}\eqfd\mathbf{\, a}_{\epsilon}+A_\epsilon\mathbf{W}+\mathbf{N}^{\epsilon}+\mathbf{Y}_{\epsilon}\label{eq:npa-ost}
\end{equation}
in the sense of equality of finite dimensional distributions and
such that for every $\delta>0$
\begin{equation}
\epsilon^{1/2-b_{\ast}+\delta}\sup_{t\in[0,T]}\|Y_{\epsilon}(t)\|\stackrel{\mathbb{P}}{\longrightarrow}0\quad\text{as}\ \epsilon\to0\label{eq:rest-ost}
\end{equation}
 where $b_{\ast}$ is given by \eqref{eq:b}.
\end{theorem}

\begin{proof}
First we will prove that $\Sigma_{1}$ is nonsingular. Let $\nu_1$ be the L\'evy measure \eqref{Rlevy} with $\epsilon=1$ and let
 \[
L=\text{lin}(\text{supp}\,\nu_{1})
\]
 be the closed linear space spanned by $\text{supp}\,\nu_{1}$. By \cite[Lemma 2.1]{Cohen07}
it suffices to show that $\text{lin}(\text{supp}\,\nu_{1})=\mathbb{R}^{d}$. Following
\cite[Corollary 4.3.5]{Jurek93} we have
\[
\text{supp}\,\nu_{1}=\{x:x=s^{B}u,\ 0\le s\le 1,\ u\in\text{supp}\,\lambda\}.
\]
 We will show
that $L$ is $B$--invariant. To this end it is enough to show that
if $x=s^{B}u\in\text{supp}\,\nu_{1}$, for some $0<s\le 1$ and $u\in\text{supp}\,\lambda$,
then $Bs^{B}u\in L$. For any $\theta \in (0,1)$,  $(\theta s)^{B}u\in\text{supp}\,\nu_{1}$ so that
\[
Bs^{B}u=\lim_{\theta\nearrow1}\frac{(\theta s)^{B}u-s^{B}u}{\log\theta}\in L.
\]
Since $L$ is closed and $B$--invariant and contains the support of $\lambda$,
$L=\mathbb{R}^{d}$ by \eqref{eq:linB}. Thus $\Sigma_{1}$ is nonsingular.

Theorem 2.2 in \cite{Cohen07} shows that the asymptotic normality
of $R_{\epsilon}(t)-\mathbb{E}[R_{\epsilon}(t)]$ holds if and only if for every $\kappa>0$ we have
\begin{equation}
\lim_{\epsilon\to0}\int_{\langle{\Sigma}_{\epsilon}^{-1}x,x\rangle>\kappa}\langle{\Sigma}_{\epsilon}^{-1}x,x\rangle\,\nu_{\epsilon}(dx)=0.\label{eq:va-b}
\end{equation}
Using \eqref{eq:sigma-sc} we have
\begin{equation*}\begin{split}
\langle{\Sigma}_{\epsilon}^{-1}s^B u,s^B u\rangle
&=\epsilon \langle (\epsilon^{-B})^{\top}\Sigma_{1}^{-1}\epsilon^{-B}s^B u, s^B u\rangle\\
&=\epsilon \langle \Sigma_{1}^{-1}\epsilon^{-B}s^B u, \epsilon^{-B} s^B u\rangle\\
&=\epsilon \langle \Sigma_{1}^{-1}(s/\epsilon)^B u, (s/\epsilon)^B u\rangle
\end{split}\end{equation*}
Note that in general $\langle Ax,x\rangle\leq \|A\| \|x\|^2 \le C \|A\| \|x\|_B^2$ (for some constant $C >0$, since all norms on $\R^{d}$ are equivalent).  Then, since $t\mapsto \|t^Bu\|_B$ is strictly increasing and $t^B x=x$ when $t=1$, the above bound shows that
\begin{equation}
\langle{\Sigma}_{\epsilon}^{-1}s^{B}u,s^{B}u\rangle\le C \epsilon \|\Sigma_{1}^{-1}\| \|(s/\epsilon)^B u\|_B^2\  \le C \epsilon \|\Sigma_{1}^{-1}\|,\label{eq:osi-lb3}
\end{equation}
whenever $0<s\le\epsilon\le1$ and $u\in S_{B}$.  Since $\Sigma_1$ is invertible we know that $c_{1}=C\|\Sigma_{1}^{-1}\|\in (0,\infty)$.  Then, for every
$\kappa>0$ and $\epsilon\in(0,1)$ we have
\begin{align*}
\int_{\langle{\Sigma}_{\epsilon}^{-1}x,x\rangle>\kappa} & \langle{\Sigma}_{\epsilon}^{-1}x,x\rangle\,\nu_{\epsilon}(dx)\\
 & =\iint_{\{(s,u)\in(0,\epsilon]\times S_{B}:\ \langle{\Sigma}_{\epsilon}^{-1}s^{B}u,s^{B}u\rangle>\kappa\}}\langle{\Sigma}_{\epsilon}^{-1}s^{B}u,s^{B}u\rangle\, s^{-2}\, ds\lambda(du)\\
 & =0\end{align*}
when $\epsilon<c_{1}^{-1}\kappa$. Indeed, in view of \eqref{eq:osi-lb3} the region
of integration is empty for $c_1\epsilon<\kappa$.
Therefore, \eqref{eq:va-b} trivially holds.

Applying \cite[Theorem 3.1]{Cohen07}
 we get \eqref{eq:npa-ost} and that
 \begin{equation}
\sup_{t\in[0,T]}\|A_{\epsilon}^{-1}Y_{\epsilon}(t)\|\stackrel{\mathbb{P}}{\longrightarrow}0\quad\text{as}\ \epsilon\to0.\label{eq:rest-ost1}
\end{equation}
 It remains to show \eqref{eq:rest-ost}.   If $\|\Sigma_{1}\|=c_2$ then $\|\Sigma_{1}^{1/2}\|=\sqrt{c_2}$.  Since every eigenvalue of $-B$ has real part less than or equal to $-b_*$, \cite[Proposition 2.2.11 (d)]{Meerschaert01} implies that for any $\delta>0$, for some $c_3>0$, we have $\|t^{-B}x\|\leq c_3 t^{-b_*+\delta}\|x\|$ for all $t\geq 1$ and all $x\in{\mathbb R}^d$.  Then $\|s^{B}\|\leq c_3 s^{b_*-\delta}$ for all $s\leq 1$.  Then for all $0<\epsilon\leq 1$ we have
 \[
\|A_{\epsilon}\|\leq \epsilon^{-1/2}\|\epsilon^{B}\|\,\|\Sigma_{1}^{1/2}\|\le c\epsilon^{-1/2-\delta+b_{\ast}}.
\]
 where $c=c_3 \sqrt{c_2 }$. Therefore,
 \[
\|Y_{\epsilon}(t)\|\le\|A_{\epsilon}\|\|A_{\epsilon}^{-1}Y_{\epsilon}(t)\|\le c\epsilon^{-1/2-\delta+b_{*}}\|A_{\epsilon}^{-1}Y_{\epsilon}(t)\|,
\]
 which together with \eqref{eq:rest-ost1} yields \eqref{eq:rest-ost}.
The proof is complete.
\end{proof}
\bigskip

\section{Simulation}\label{sec:sim}

The main goal of this paper is to provide a practical method for simulating the sample paths of an operator stable L\'evy process $\X$.  Theorem \ref{th:ost} decomposes $\X$ into the drift $a_{\epsilon}t$, the large jumps $N^{\epsilon}(t)$, and a Gaussian approximation of the small jumps, with a remainder term whose supremum converges to zero in probability at a polynomial rate as the number of large jumps increases (or, equivalently, as the size of the remaining jumps tends to zero).  In this section, we will demonstrate the practical application of Theorem \ref{th:ost}, and illustrate the resulting sample paths.

Theorem \ref{th:ost} justifies the use of the process
\begin{equation}\label{eq:Ze}
Z_{\epsilon}(t):=a_{\epsilon}t+A_{\epsilon}W(t)+N^{\epsilon}(t),	
\end{equation}
with $A_{\epsilon}$ given by \eqref{Aepsdef} and $W(t)$ a standard Brownian motion, to simulate sample paths of the operator stable process $\{X(t)\}_{t\in[0,T]}$ specified
by \eqref{eq:ch.f.} and \eqref{eq:Lm-ost}.  Formula \eqref{eq:rest-ost}
shows that the approximation converges faster when the real parts of the
eigenvalues of $B$ are uniformly larger.  Remark 7.2.10 in \cite{Meerschaert01} shows that the real parts of the eigenvalues of the exponent $B$ govern the tails of the operator stable process $X(t)$, and $b_*=\min\{b_1,\ldots,b_d\}>1/2$ determines the lightest tail, in the sense that ${\mathbb E}|\langle X(t),u \rangle | ^\rho$ diverges for all $\rho>1/b_*$ and all $u\neq 0$.  Hence the convergence is faster when $X$ has a heavier tail.

The process $Z_{\epsilon}(t)$ in \eqref{eq:Ze} approximates the operator stable process $\{X(t)\}_{t\in[0,T]}$ by discarding small jumps, replacing their sum by an appropriate Brownian motion with drift.  The discarded random jumps are all of the form $r^B v$ where $v\in S_B$ and $r\leq \epsilon$.  If $B$ has no nilpotent part then $\|r^B v\|_B\leq \epsilon^{b_*}$.  Hence in order to retain all jumps larger than $m$ it suffices to take $\epsilon=m^{1/b_*}$, and then the number of jumps simulated will be Poisson with mean $m^{-1/b_*}T\lambda(S_B)$.  If there is a nilpotent part, the bound involves additional $\log\epsilon$ terms.

In general, an operator stable process can be decomposed into two independent component processes, one Gaussian and another having no Gaussian component.  The two components are supported on subspaces of ${\mathbb R}^d$ whose intersection is trivial.  In practical applications, Theorem \ref{th:ost} is applied to the nonnormal component.  In the case where $X(t)$ has both a normal and a nonnormal component, the resulting approximation combines a full dimensional Brownian motion with drift, and a Poissonian component restricted to the nonnormal subspace.  For the remainder of this section, we will focus on simulating operator stable laws on $\R^2$ having no normal component.

In practical applications, it is advantageous to produce a simulated process whose mean (if the mean exists) equals that of the operator stable process $X(t)$.  If every eigenvalue of the exponent $B$ has real part $b<1$, then the mean exists, by \cite[Theorem 8.2.14]{Meerschaert01}. If any eigenvalue has real part $b> 1$ then the mean is undefined.  In the former case, one can choose $a_{\epsilon}$ so that the right-hand side in  \eqref{eq:Ze} has mean zero.  Recall that the number of terms $M_\epsilon$ in the sum \eqref{eq:ostN} defining $N^{\epsilon}(t)$ is Poisson with mean $\theta_\epsilon=T\lambda(S_B)/\epsilon$, and that conditional on $M_\epsilon=n$, $(\Gamma_1/\theta_{\epsilon},\ldots,\Gamma_n/\theta_{\epsilon})$ is equal in distribution to the vector of order statistics from $n$ IID standard uniform random variables.  Condition to get
$\E[N^{\epsilon}(t)|M_\epsilon=n]=n(t/T) \E[(\epsilon U)^{-B}]\E[v]$
where $U$ is standard uniform and $v$ has distribution $\lambda/\lambda(S_{B})$.  Removing the condition and simplifying shows that
\begin{equation}\label{center}
\E[N^{\epsilon}(t)]=t\lambda(S_B)\epsilon^{B-I} \E[U^{-B}]\E[v].
\end{equation}
Since $\E[W(t)]=0$ we can set $a_{\epsilon}t=-\E[N^{\epsilon}(t)]$ to get mean zero.  Note that for such $B$ we have $\|\epsilon^{B-I}x\|\to\infty$ for all $x\neq 0$ by \cite[Theorem 2.2.4]{Meerschaert01}, so that $\|a_{\epsilon}\|\to\infty$ as $\epsilon\to 0$.  This reflects the fact that, in the finite mean case, the infinite series \eqref{eq:series} does not converge without centering.  Finally we note that, if $\E[v]=0$, then no centering is necessary.

In this section, we assume a fixed coordinate system on $\R^2$ with the standard coordinate vectors $e_1=[1,0]^{\top}$ and $e_2=[0,1]^{\top}$, and we write $X(t)=X_1(t)e_1+X_2(t)e_2$.  Recall that a strictly operator stable process satisfies the scaling relationship
\begin{equation}\label{osscale2}
X(t)\dlim t^B X(1)
\end{equation}
 for all $t>0$.  All plots in this section use $T=1$ and $\epsilon=0.001$, and we show the simulated processes at the time points $t=n \Delta t$ for $0\leq t\leq T$ with $\Delta t=0.001$.  Unless otherwise noted, we use the standard Euclidean norm.

\begin{figure}\label{figexA}
\begin{center}
\includegraphics[width=2.9in]{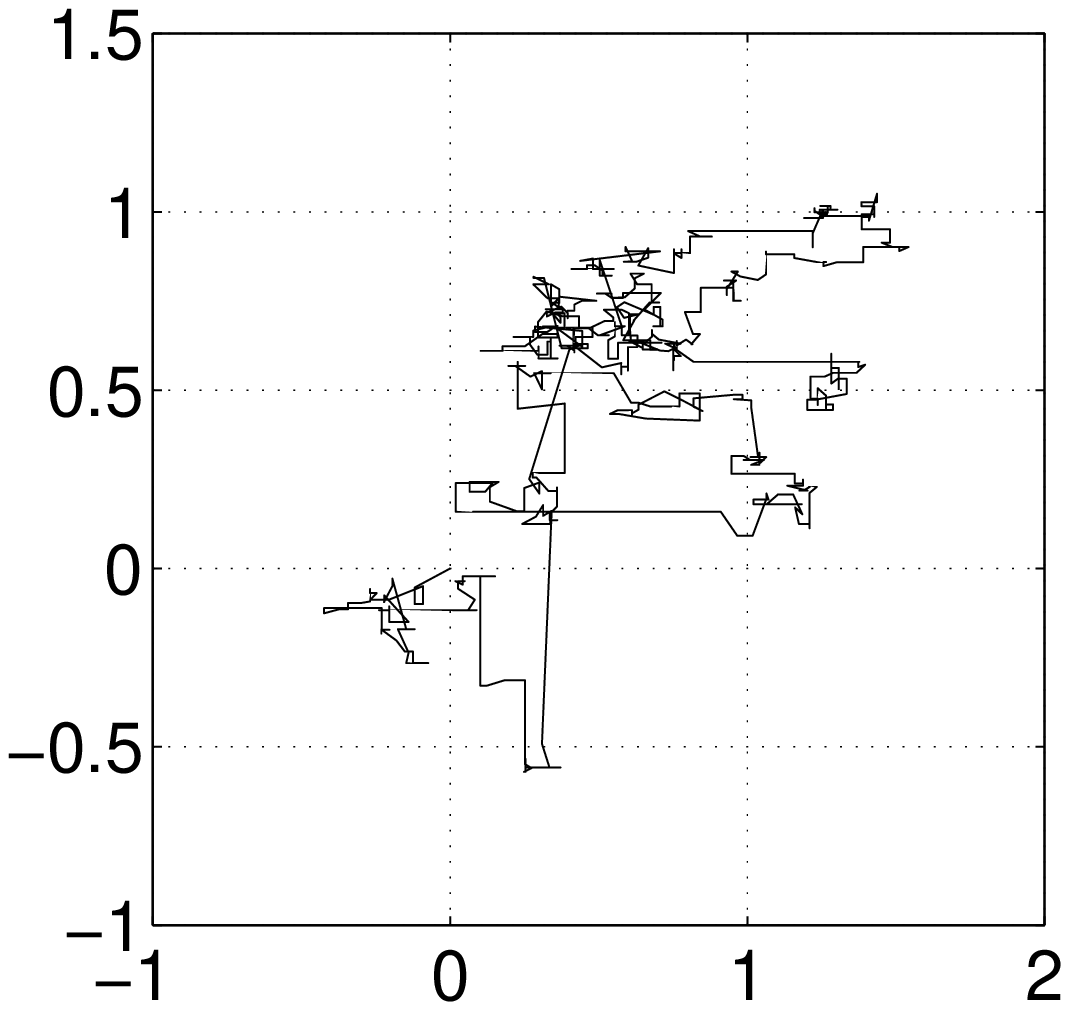}
\includegraphics[width=2.9in]{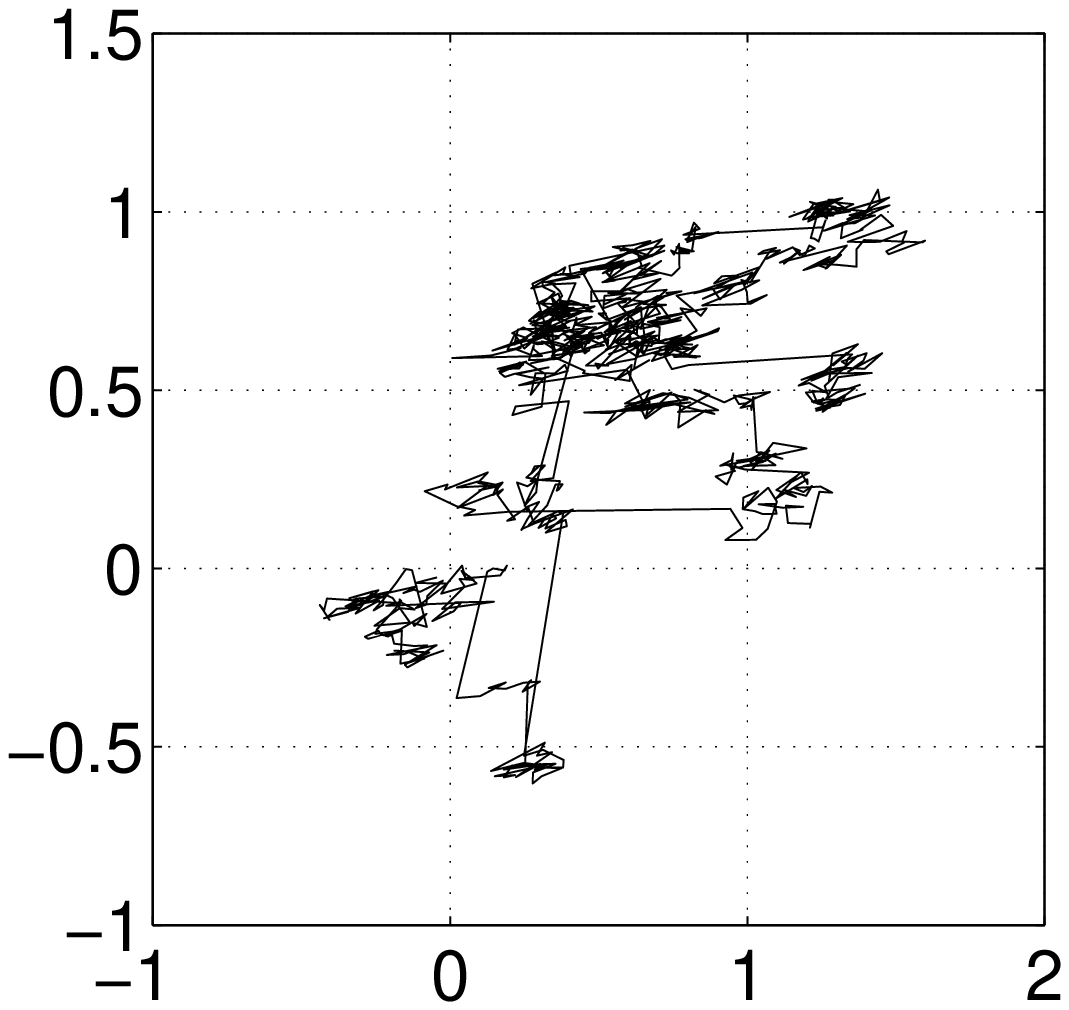}
\includegraphics[width=2.9in]{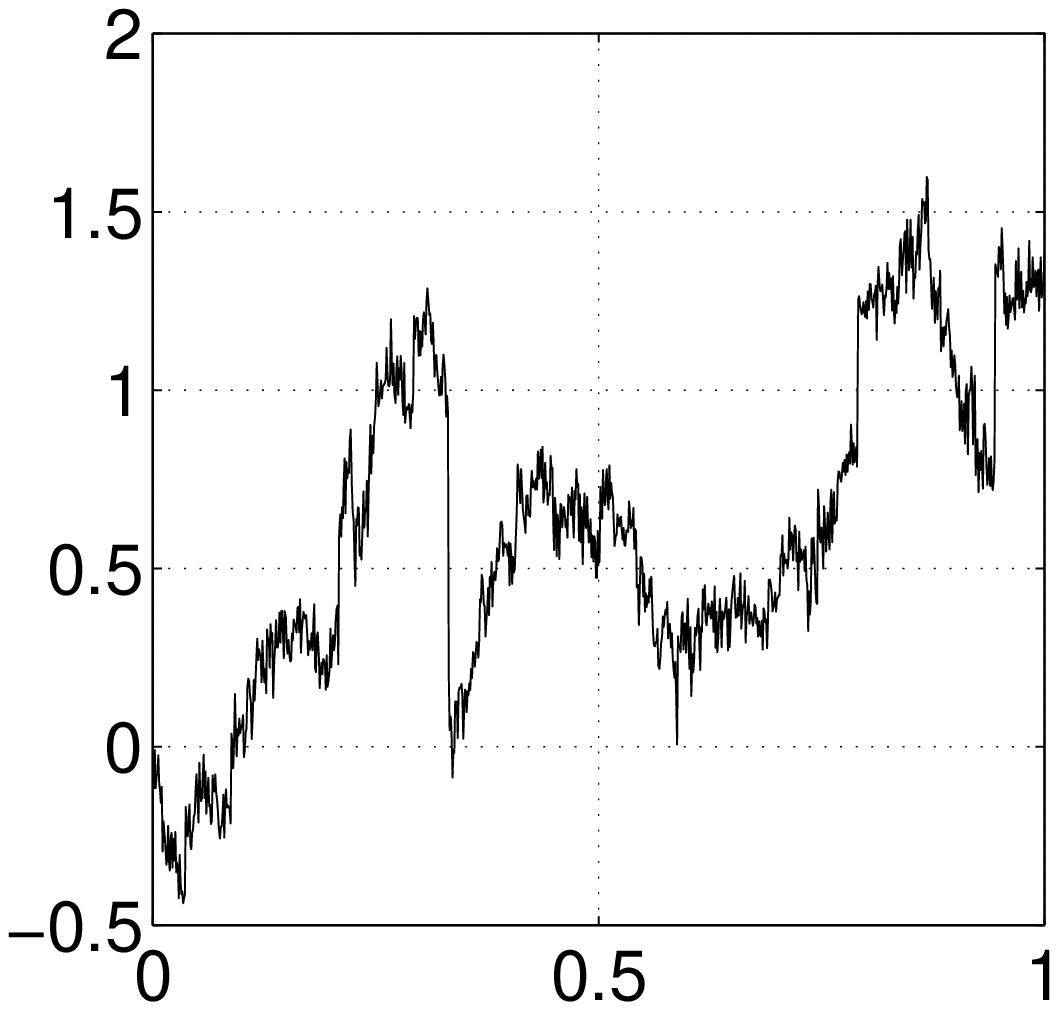}
\includegraphics[width=2.9in]{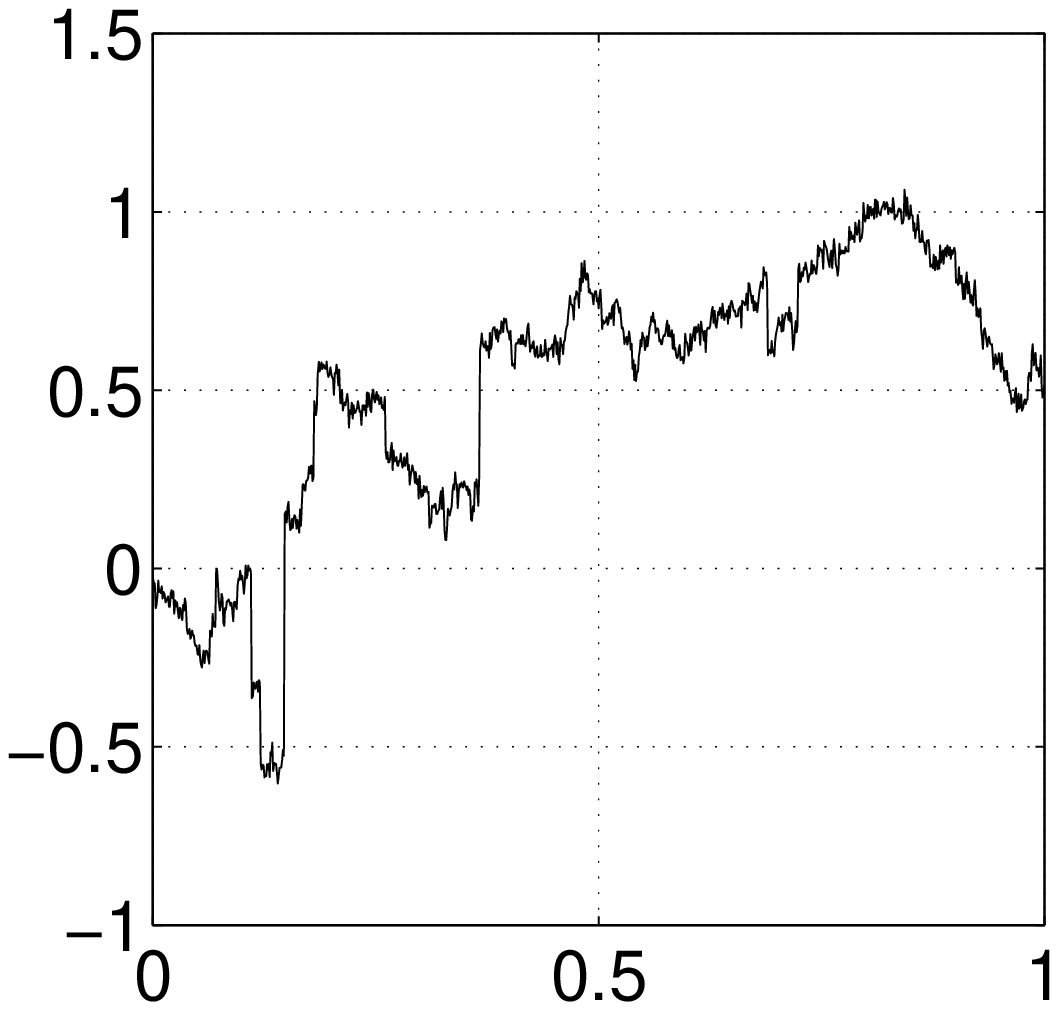}
\caption{Simulated operator stable process for Example \ref{ExA}, with independent symmetric stable marginals.  Top left panel shows the sample path of the shot noise process $M^\epsilon(t)$, and top right panel shows the corresponding operator stable process $X(t)$.  Bottom left panel shows the marginal process $X_1(t)$, and bottom right panel shows $X_2(t)$. }
\end{center}
\end{figure}

\begin{example}\label{ExA}  Equation \eqref{eq:Ze} was used to simulate an operator stable process $X(t)$ whose exponent is diagonal
\[B=\left[\begin{array}{cc} 1/1.8 & 0 \\0 & 1/1.5\end{array}\right] = {\rm diag}(b_1,b_2)\]
so that $Be_i=b_ie_i$ with $b_1=1/1.8$ and $b_2=1/1.5$.
Since the exponent is already in Jordan form, we can take $\|x\|_B$ to be the usual Euclidean norm, so that $S_B$ is the unit circle.  We choose the spectral measure $\lambda$ to place equal masses of $1/4$ at the four points $\pm e_1$ and $\pm e_2$.  Then $\E[v]=0$ in \eqref{center} so that no centering is needed, as the simulated process has mean zero without any centering.   Then $\Lambda={\rm diag}(1/2,1/2)$, $\Sigma_1={\rm diag}(9/2,3/2)$, and $A_\epsilon={\rm diag}(3\sqrt {5}\sqrt [3]{10}/10,\sqrt {15}/10)$.  It is easy to see from the definition $t^B=I+B\log t+(B\log t)^2/2!+\cdots$ that $t^B={\rm diag}(t^{b_1},t^{b_2})$.  From the scaling relation \eqref{osscale2} it follows that
\[X_i(t)\dlim t^{b_i}X_i(1) .\]
  Hence the coordinate marginals are (strictly) stable with index $\alpha_1=1/b_1=1.8$ and $\alpha_2=1.5$, respectively.
The top right panel in
Figure 1 
shows a typical sample path of the process, an irregular meandering curve punctuated by occasional large jumps.  The top left panel shows the corresponding shot noise part $N^\epsilon(t)$ before the Gaussian approximation of the small jumps is added.  Since the spectral measure is concentrated on the coordinate axes, the large jumps apparent in the sample path of
Figure 1 
are all either horizontal or vertical.
Pruitt and Taylor \cite{PT69} showed that the Hausdorff dimension of the sample path is $\max\{\alpha_1,\alpha_2\}=1.8$ with probability one.  Since the spectral measure is concentrated on the coordinate axes, Lemma 2.3 in Meerschaert and Scheffler \cite{MAcorr} shows that the coordinates $X_1(t)$ and $X_2(t)$ are independent stable processes.
The bottom panels in
Figure 1 
graph each marginal process.  Note that the large jumps occur at different times, reflecting the independence of the marginals.  Blumenthal and Getoor \cite{BG62} showed that the graph of the stable process $X_i(t)$ has Hausdorff dimension $2-\alpha_i$.
\end{example}

\begin{figure}\label{figexAs}
\begin{center}
\includegraphics[width=2.9in]{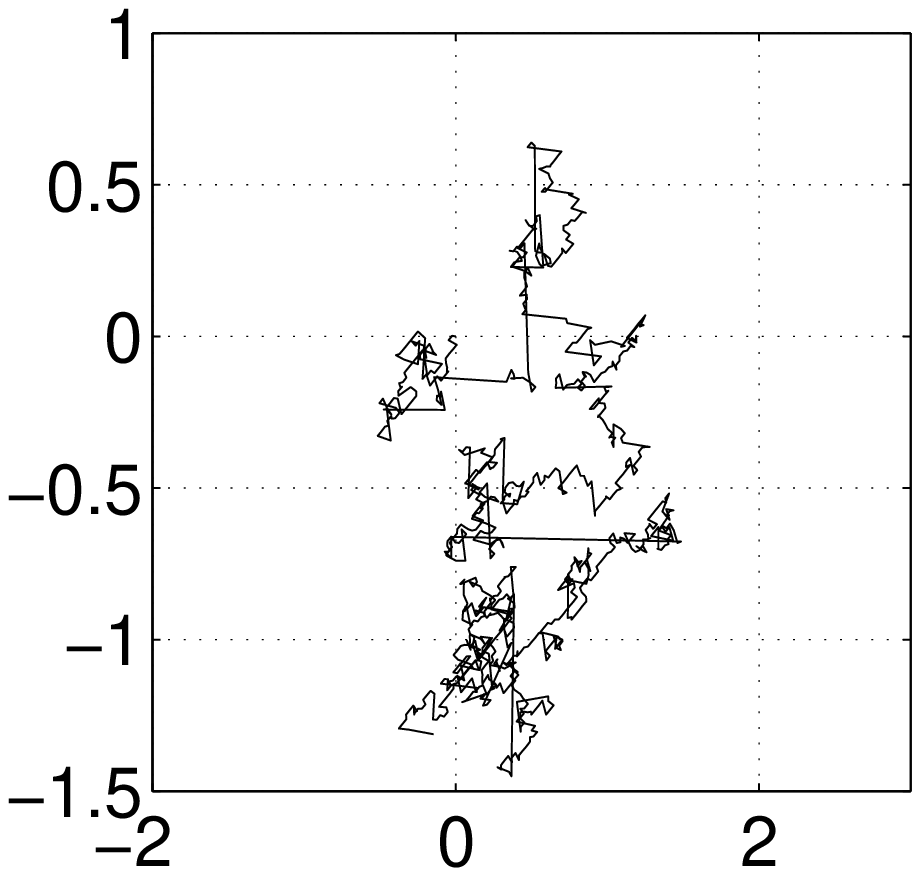}
\includegraphics[width=2.9in]{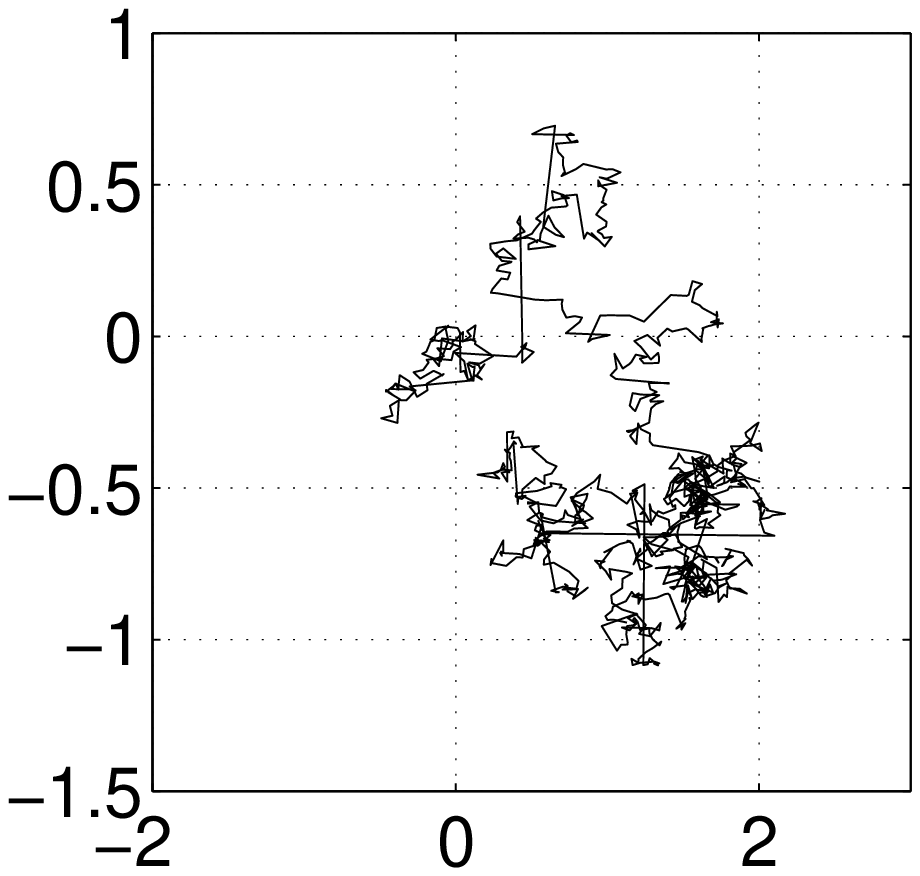}
\includegraphics[width=2.9in]{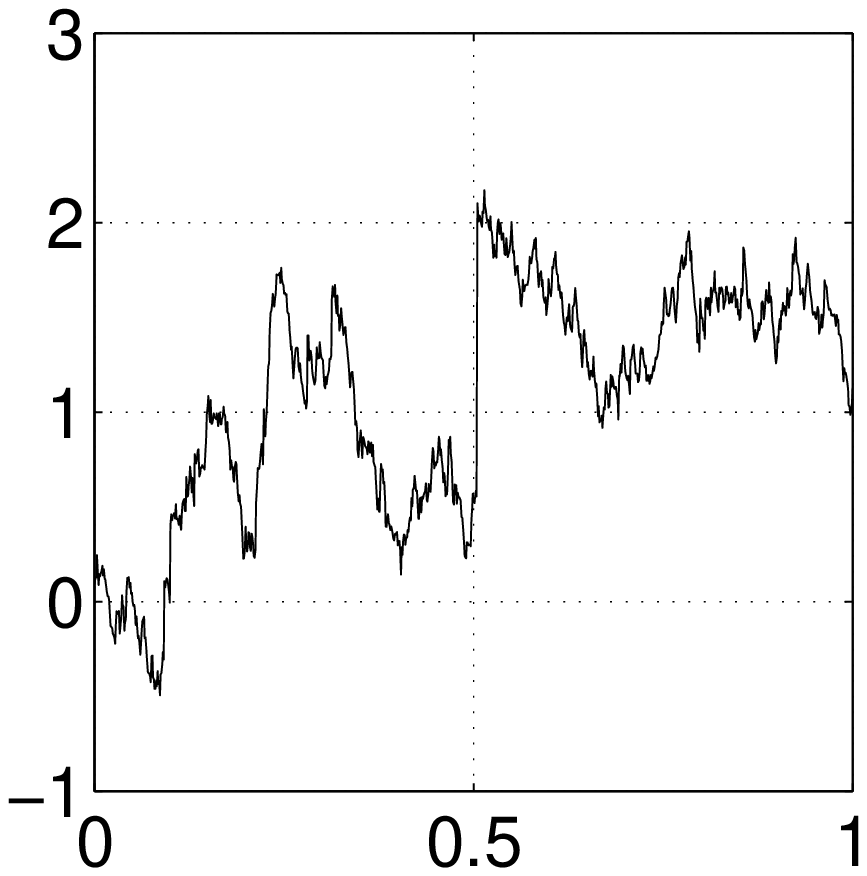}
\includegraphics[width=2.9in]{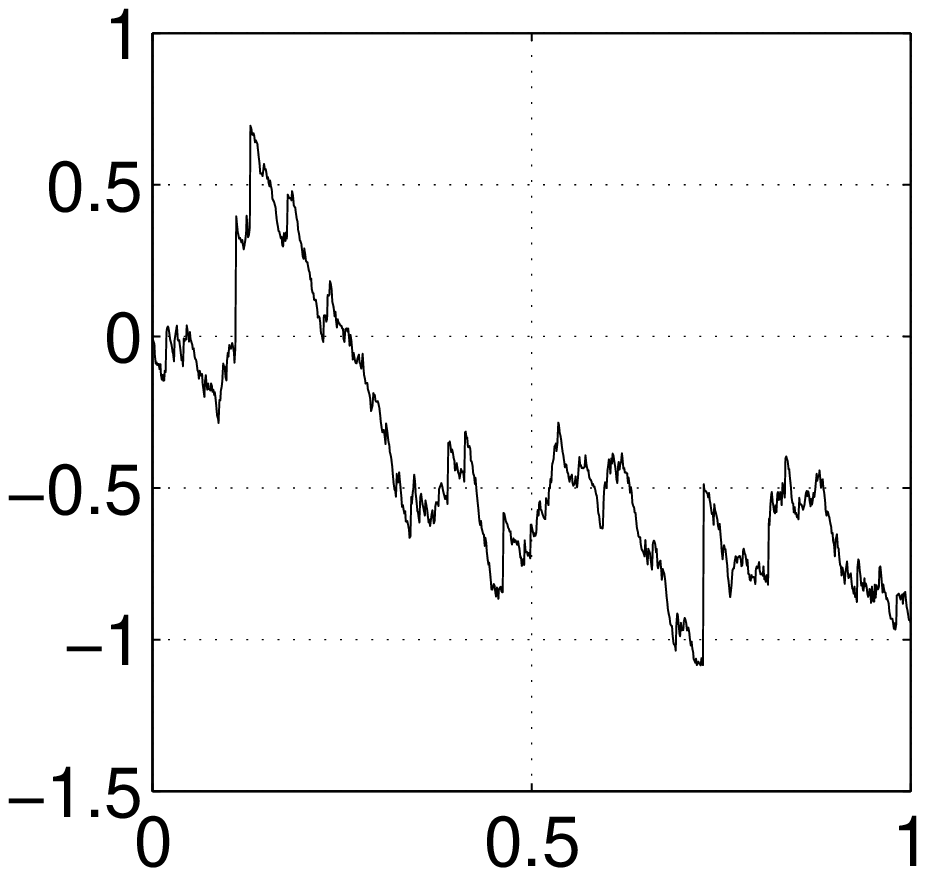}
\caption{Simulated operator stable process for Example \ref{ExAs}, with independent skewed stable marginals.  Top left panel shows the sample path of the shot noise process $M^\epsilon(t)$, and top right panel shows the corresponding operator stable process $X(t)$.  Bottom panels show the marginal processes $X_1(t)$ and $X_2(t)$.}
\end{center}
\end{figure}

\begin{example}\label{ExAs}    The same exponent $B$ is used as in Example \ref{ExA}, but now we take the spectral measure $\lambda(e_i)=1/2$.  The matrices $\Lambda$ and $A_\epsilon$ turn out to be the same as Example \ref{ExA}.  The marginals $X_i(t)$ are still stable with index $\alpha_1=1.8$ and $\alpha_2=1.5$, but they are no longer symmetric, and we center to zero expectation.  From \eqref{center} we get $a_{\epsilon}=[45\sqrt [3]{10}/4, 15]^\top$ to compensate the shot noise portion to mean zero.
Figure 2 
shows a typical sample path and component graphs for this process.  Since the spectral measure is concentrated on the positive coordinate axes, the large jumps apparent in the component graphs are all positive.
\end{example}

\begin{figure}\label{figexA1}
\begin{center}
\includegraphics[width=2.9in]{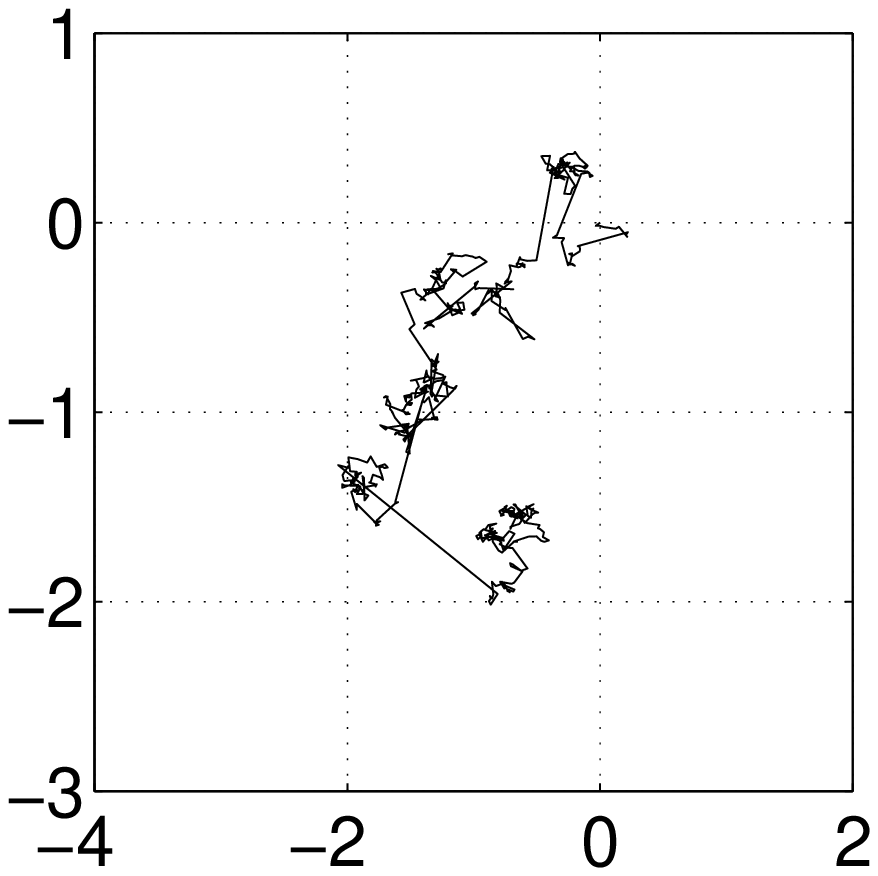}
\includegraphics[width=2.9in]{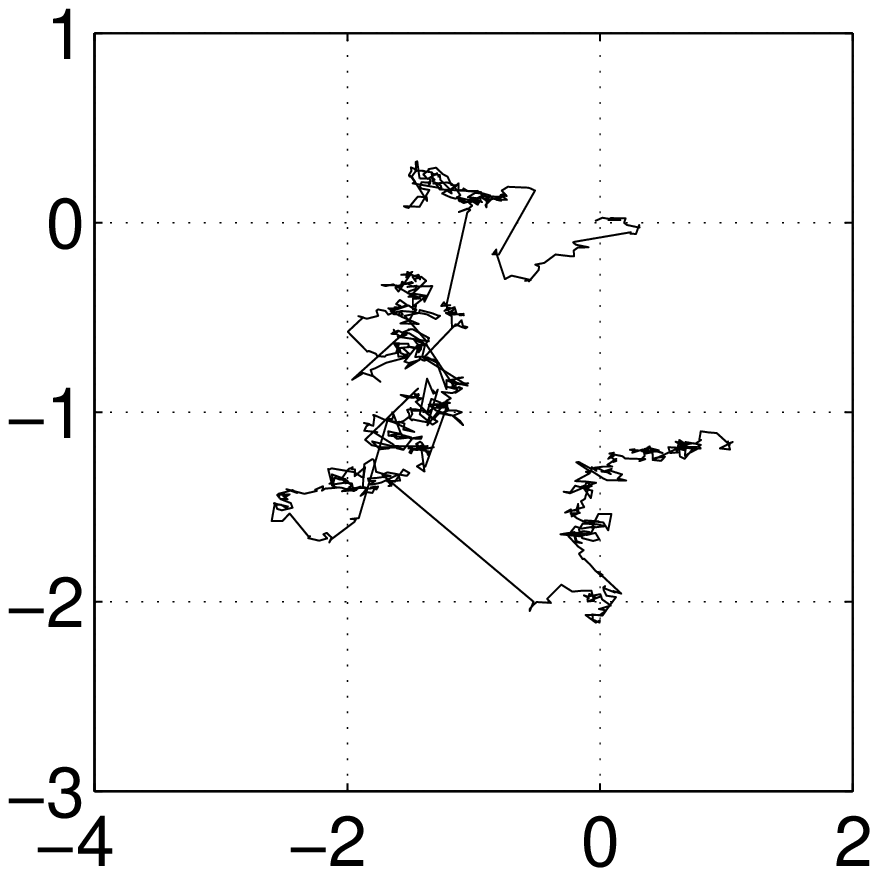}
\includegraphics[width=2.9in]{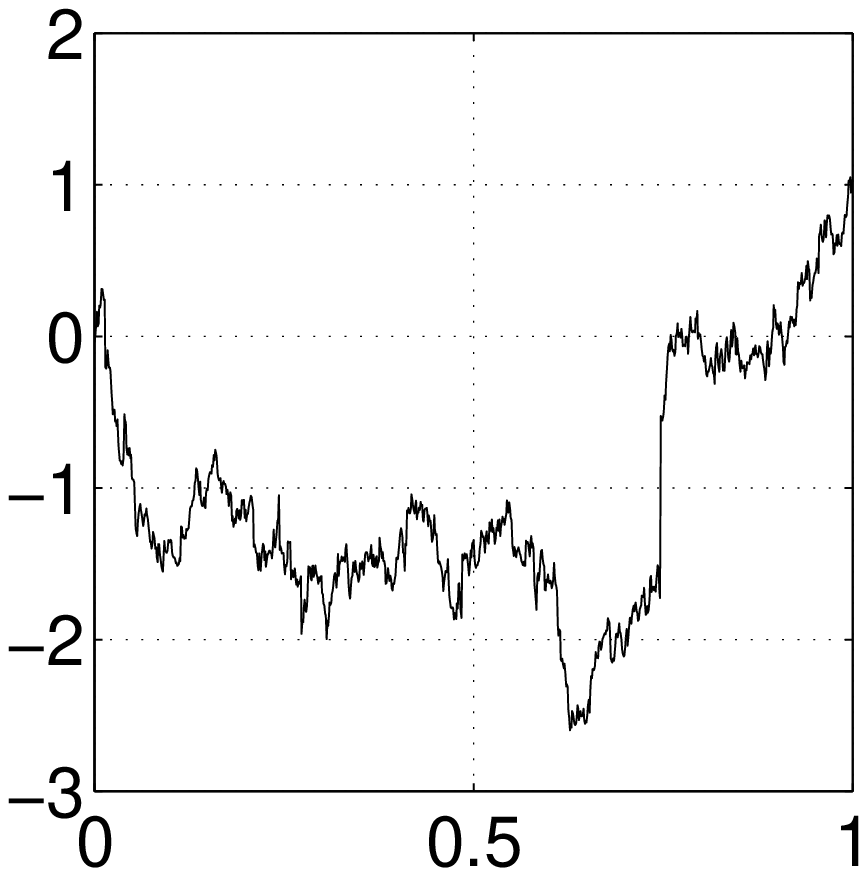}
\includegraphics[width=2.9in]{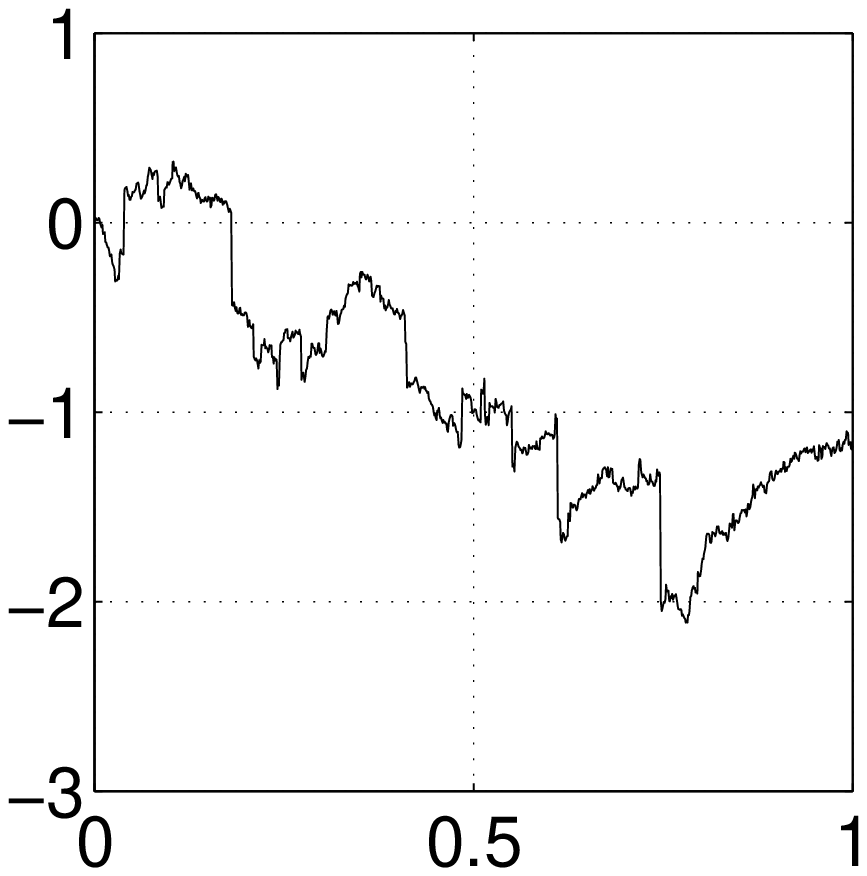}
\caption{Simulated operator stable process for Example \ref{ExA1}, with dependent stable marginals.   Top left panel shows the sample path of the shot noise process, and top right panel shows the corresponding operator stable process.   Bottom panels show the marginal processes.}
\end{center}
\end{figure}

\begin{example}\label{ExA1}    The same exponent $B$ is used as in Example \ref{ExA}, but now we take the spectral measure $\lambda$ to be uniformly distributed on the unit circle: set $v=(x^2+y^2)^{-1/2}[x,y]^\top$ where $x,y$ are independent standard normal.  The matrices $\Lambda$ and $A_\epsilon$ turn out to be the same as Example \ref{ExA}.  Since $\E[v]=0$, no centering is needed.  The marginals $X_i(t)$ are symmetric stable with index $\alpha_1=1.8$ and $\alpha_2=1.5$, but they are no longer independent.  The top panels in
Figure 3 
show a typical sample path of the process.  Since the spectral measure is uniform, the large jumps apparent in the sample path take a random orientation.  Theorem 3.2 in Meerschaert and Xiao \cite{MarkXiao} shows that the sample path is a random fractal, a set whose Hausdorff and packing dimension are both equal to $1.8$ with probability one.  The bottom panels in
Figure 3 
show the graphs of each marginal process.  Note that the large jumps in both marginals are simultaneous, reflecting the dependence.
\end{example}

\begin{figure}\label{figexB}
\begin{center}
\includegraphics[width=4in]{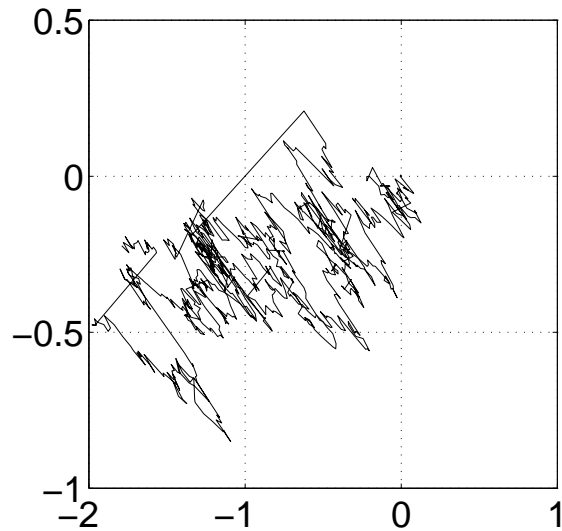}
\caption{Simulated operator stable sample path for Example \ref{ExB}, with independent skewed stable components along nonstandard coordinate axes.}
\end{center}
\end{figure}

\begin{example}\label{ExB}    Figure 2 of Zhang et al. \cite{RW2d} represents a model of contaminant transport in fractured rock.  Pollution particles travel along fractures in the rock, which form at specific angles due to the geological structure of the rock matrix.  An operator stable process $X(t)$ represents the path of a pollution particle, with independent skewed stable components in the fracture directions.  The skewness derives from the fact that particles jump forward (downstream) when mobilized by water that flows through the fractured rock.  The two components of $X(t)$ are skewed stable with index $\alpha=1.3$ on the line with angle $\theta_1=30^\circ$ measured from the positive $e_1$ axes as usual, and index 1.7 on the line with angle $\theta_2=-35^\circ$.  The two stable laws are independent.  The $e_1$ axis represents the overall direction of flow, caused by a differential in hydraulic head (pressure caused by water depth).  The exponent $B$ has one eigenvalue $b_1=1/1.3$ with associated eigenvector $v_1=R_{\theta_1}e_1=[.865,.500]^\top$, and another eigenvalue $b_2=1/1.7$ with associated eigenvector $v_2=R_{\theta_2}e_1=[.820,-.572]^\top$. The spectral measure is specified as $\lambda(v_1)=0.4$ and $\lambda(v_2)=0.6$, representing the relative fraction of jumps along each fracture direction.  In order to compute the matrix power $t^B$ a change of basis is useful.  Define the matrix $P$ according to $Pe_i=v_i$ so that
\[P=\left[\begin{array}{cc} .865 & .820 \\ .500 & -.572\end{array}\right]\]
and $D=P^{-1}BP={\rm diag}(b_1,b_2)$ is a diagonal matrix.  Then the exponent
\[B=PDP^{-1}=\left[\begin{array}{cc} .688& .142 \\.057 & .671\end{array}\right] .\]
From \eqref{eq:L} we get
\[\Lambda=\left[\begin{array}{cc} .703& -.109 \\-.109 & .297\end{array}\right] .\]
Since $t^D={\rm diag}(t^{b_1},t^{b_2})$ we can compute $t^B=Pt^DP^{-1}$ and integrate in \eqref{eq:s1} to get the Gaussian covariance matrix $\Sigma_\epsilon$
whose symmetric square root is given by
\[A_\epsilon= \left[ \begin {array}{cc}  0.723&- 0.416\\\noalign{\medskip}- 0.416&
 0.407\end {array} \right]\]
To compute the square root, we decompose $\Sigma_\epsilon=QEQ^{-1}$ where $E={\rm diag}(c_1,c_2)$, $c_i$ are the eigenvalues of $\Sigma_\epsilon$, and the columns of $Q$ are the corresponding eigenvectors, so that $A_\epsilon=QE^{1/2}Q^{-1}$ where $E^{1/2}={\rm diag}(c_1^{1/2},c_2^{1/2})$.
From \eqref{center} we get $a_{\epsilon}=[27.9, -10.1]^\top$ to compensate the shot noise portion to mean zero.
Note that $B^\top u_i=b_iu_i$ where $u_1=[.572,.820]^\top$ and $u_2=[.500,-.865]^\top$ are the dual basis vectors.  Then each projection $\ip{ X(t)}{u_i}$ is (strictly) stable with index $\alpha_i=1/b_i$, since
\[\ip{X(t)}{u_i}\dlim\ip{t^BX(1)}{u_i}=\ip{X(1)}{t^{B^\top} u_i}=\ip{X(1)}{t^{b_i} u_i}=t^{b_i}\ip{X(1)}{u_i} .\]
Hence $.572 X_1(t)+.820 X_2(t)$ is stable with index $\alpha_1=1.3$ and $.500 X_1(t)-.865 X_2(t)$ is stable with index $\alpha_2=1.7$.  Lemma 2.3 in \cite{MAcorr} shows that these two skewed stable marginals of $X(t)$ are independent, since the spectral measure is concentrated on the eigenvector coordinate axes $\ip{x}{v_i}=0$.
Figure 4 
shows a typical sample path, along with the coordinate marginals.  Note that the large jumps lie in the $v_i$ directions.  The mean zero operator stable process $X(t)$ represents particle location in a moving coordinate system, with origin at the center of mass.  Hence
Figure 4 
illustrates the dispersion of a typical pollution particle away from the center of mass of the contaminant plume.  Dispersion is the spreading of particles due to variations in velocity, and it is the main cause of plume spreading in ground water hydrology.
\end{example}

\begin{figure}\label{figexN}
\begin{center}
\includegraphics[width=4in]{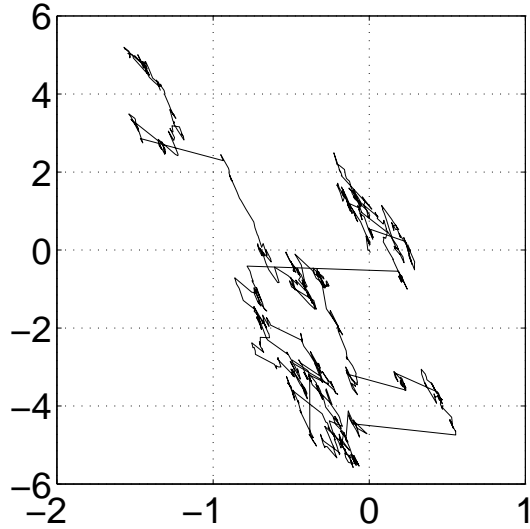}
\caption{Simulated operator stable sample path for Example \ref{ExN}, whose exponent has a nilpotent part.}
\end{center}
\end{figure}

\begin{example}\label{ExN}  We simulate an operator stable process $X(t)$ whose exponent has a nilpotent part
\[B=\left[\begin{array}{cc} 1/1.5 & 0 \\1 & 1/1.5\end{array}\right]\]
We choose the spectral measure $\lambda$ to place equal masses of $1/4$ at the four points $\pm e_1$ and $\pm e_2$.  Then $\E[v]=0$ in \eqref{center} so that no centering is needed.   Here $\Lambda={\rm diag}(1/2,1/2)$,
\[\Sigma_1=\left[\begin{array}{cc} 3/2 & -9/2 \\-9/2 & 57/2\end{array}\right] \]
and
\[A_\epsilon= \left[ \begin {array}{cc}  0.146&- 0.359\\\noalign{\medskip}-
 0.359& 4.009\end {array} \right] .\]
Note that $t^B=t^bt^N$ where $b=1/1.5$ and
\[t^N=\left[\begin{array}{cc} 1 & 0 \\ \log t & 1\end{array}\right] .\]
From \eqref{osscale2} it follows that the second marginal $X_2(t)$ is symmetric stable with index $\alpha=1/b=1.5$.  The first marginal is not stable, but it lies in the domain of attraction of a symmetric stable with index $\alpha=1.5$, see \cite[Theorem 2]{marginal}.
Figure 5 
shows a typical sample path of the process.  The large jumps apparent in the sample path of
Figure 5 
are all of the form $t^Bv$ where $v=\pm e_i$ and $t>0$.  Hence they are either vertical, or they lie on the curved orbits $\pm t^Be_1$.
Theorem 3.2 in \cite{MarkXiao} shows that the sample path is almost surely a random fractal with dimension $1.5$.  Lemma 2.3 in \cite{MAcorr} shows that the coordinates $X_1(t)$ and $X_2(t)$ are not independent.
\end{example}

\begin{figure}\label{figexX}
\begin{center}
\includegraphics[width=4in]{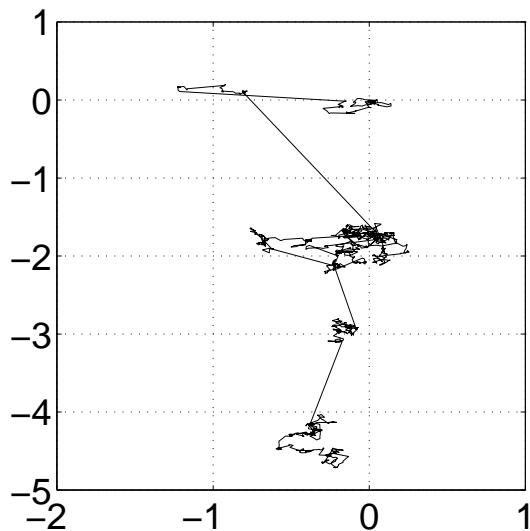}
\caption{Simulated operator stable sample path for Example \ref{ExX}, whose exponent has complex eigenvalues.}
\end{center}
\end{figure}

\begin{example}\label{ExX}  We simulate an operator stable process $X(t)$ whose exponent
\[B=\left[\begin{array}{cc} 1/1.5 & 1 \\-1 & 1/1.5\end{array}\right]\]
has complex eigenvalues $b\pm i$ with $b=1/1.5$.  We choose the spectral measure $\lambda$ to place equal masses of $1/4$ at the four points $\pm e_1$ and $\pm e_2$, so that $\E[v]=0$ in \eqref{center} and no centering is needed.   Here $\Lambda={\rm diag}(1/2,1/2)$, and
\[A_\epsilon=  \left[ \begin {array}{cc}  0.387& 0.0\\\noalign{\medskip} 0.0& 0.387
\end {array} \right] .\]
In this case, $t^B=t^b R_{\theta(t)}$ with $\theta(t)=\ln t$, since we can write $B=bI+K$ where the matrix exponential $\exp(sK)=R_s$.
The coordinate marginals $X_1(t)$ and $X_2(t)$ are not stable, but they are both semistable with index $\alpha=1/b=1.5$, see \cite[Theorem 2]{marginal}.  Lemma 2.3 in \cite{MAcorr} shows they are not independent.
Figure 6 
shows a typical sample path of the process.  The large random jumps
are of the form $t^Bv$ where $v=\pm e_i$, so that the angle varies along with the length of the jump.  The sample path is a fractal with dimension $1.5$, see \cite[Theorem 3.2]{MarkXiao}.
\end{example}

\begin{figure}\label{figexC}
\begin{center}
\includegraphics[width=2.9in]{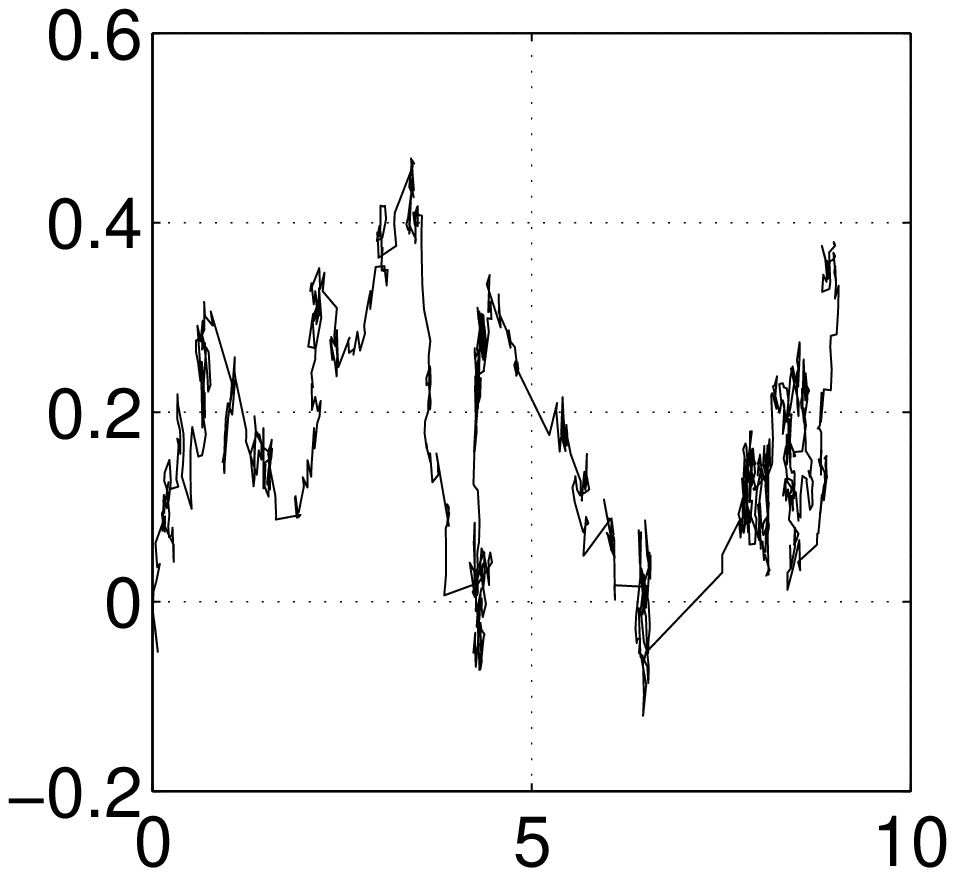}
\includegraphics[width=2.9in]{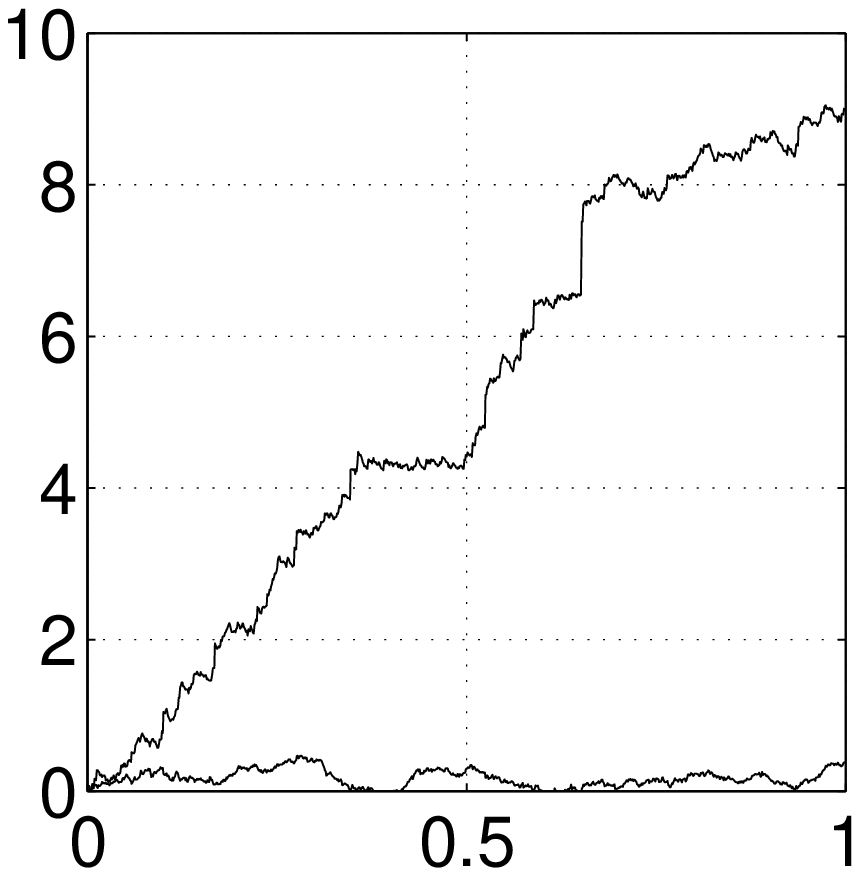}
\caption{Simulated operator stable process for Example \ref{ExC}, modeling a pollution particle moving through underground water in a heterogeneous porous medium consisting of sand, gravel, and clay. Left panel depicts the sample path of a moving particle.  Right panel shows the coordinate marginals.}
\end{center}
\end{figure}

\begin{example}\label{ExC} Figure 1 in Zhang et al. \cite{RW2d} presents an operator stable model $X(t)$ with diagonal exponent
\[B=\left[\begin{array}{cc} 1/1.5 & 0 \\0 & 1/1.9\end{array}\right] \]
and spectral measure $\lambda$ that places masses of 0.3 at $e_1$, 0.2 at $\pm 6^\circ$, 0.1 at $\pm 12^\circ$, and 0.05 at $\pm 18^\circ$ on the unit sphere in the standard Euclidean norm.  Large jumps are along the positive $x$-axis, or along the orbits $t^Bu$ where $u$ is a unit vector at $\pm 6^\circ$, $\pm 12^\circ$, or $\pm 18^\circ$, representing displacements of a pollutant particle in an underground aquifer with a mean flow in the positive $x$ direction, but some dispersion due to the intervening porous medium.  The average plume velocity is $v=[10,0]^\top$ so that $\E[X(t)]=tv$.
Figure 7 
depicts the path of a typical particle.
Here \[\Lambda=\left[ \begin {array}{cc}  0.977& 0.0\\\noalign{\medskip} 0.0& 0.0226\end {array} \right]
\]
and
\[A_\epsilon=   \left[ \begin {array}{cc}  0.541& 0.0\\\noalign{\medskip} 0.0&
 0.546\end {array} \right] .\]
From \eqref{center} we compute $a_{\epsilon}=[29.7, 0]^\top$ and, in the simulation code, we first center to mean zero, and then add the mean velocity.
\end{example}


\begin{figure}\label{figexE}
\begin{center}
\includegraphics[width=4.0in]{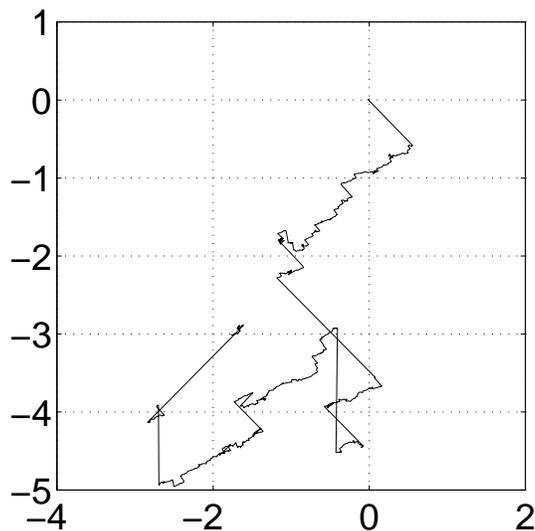}
\caption{Simulated operator stable sample path for Example \ref{ExE}, modeling the motion of a pollution particle moving through water in fractured rock. The mean zero sample path represents deviation from the plume center of mass.}
\end{center}
\end{figure}

\begin{example}\label{ExE} This example follows the transport model number 22 for contaminant transport in complex fracture networks from Reeves et al. \cite{Reeves2}.  The exponent $B$ has eigenvectors $v_1=[\sqrt{2}/2,\sqrt{2}/2]^\top$ and $v_2=[\sqrt{2}/2,-\sqrt{2}/2]^\top$ at $+ 45^\circ$ and $-45^\circ$ on the unit circle with eigenvectors $b_1=1/1.1$ and $b_2=1/1.2$ respectively.  Writing $Pe_i=v_i$ we get $D=P^{-1}BP={\rm diag}(b_1,b_2)$ so that
 \[B= PDP^{-1}=\left[ \begin {array}{cc} {{115}/{132}}&{{5}/{132}}
\\\noalign{\medskip}{{5}/{132}}&{{115}/{132}}\end {array}
 \right]
\]
Since $t^D={\rm diag}(t^{b_1},t^{b_2})$ we also have
 \[t^B= P t^D P^{-1}=\left[ \begin {array}{cc} (t^{b_1}+t^{b_2})/2&(t^{b_1}-t^{b_2})/2\\
 \noalign{\medskip}(t^{b_1}-t^{b_2})/2&(t^{b_1}+t^{b_2})/2\end {array} \right]
\]
The spectral measure has weights $0.4$ and $\pm 45^\circ$ and 0.2 at $e_2$.  The L\'evy measure is concentrated on the two straight line orbits $\{t^Bv_i:t>0\}$ and on the curved orbit $\{t^Be_2:t>0\}$.  Marginals $\ip{X(t)}{v_i}$ are stable with index $\alpha_1=1.1$ and $\alpha_2=1.2$ respectively, but they are not independent, since the spectral measure is not concentrated on the eigenvector axes.  The first marginal process $\ip{X(t)}{v_1}$ is positively skewed, since the projection of the L\'evy measure onto the first eigenvector coordinate places all mass on the positive half line.  The second marginal $\ip{X(t)}{v_2}$ is the sum of two independent stable processes, one with positive skewness resulting from the $v_2$ orbit, and one with negative skewness resulting from the projection of the $e_2$ orbit onto the negative $v_2$ axis. As in Example \ref{ExB} we compute $\Lambda={\rm diag}(0.4,0.6)$ and
\[A_\epsilon=\left[ \begin {array}{cc}  0.0603&- 0.0204\\\noalign{\medskip}- 0.0203& 0.0723\end {array} \right] .\]
From \eqref{center} we compute $a_{\epsilon}=[11.35, 4.43]^\top$ to correct the shot noise process to mean zero.
Figure 8 
shows a typical sample path.  In this case, the sample path represents the growing deviation of a typical pollution particle from the plume center of mass.
\end{example}

\begin{figure}\label{figex2a}
\begin{center}
\rotatebox{270}{\includegraphics[width=3.0in]{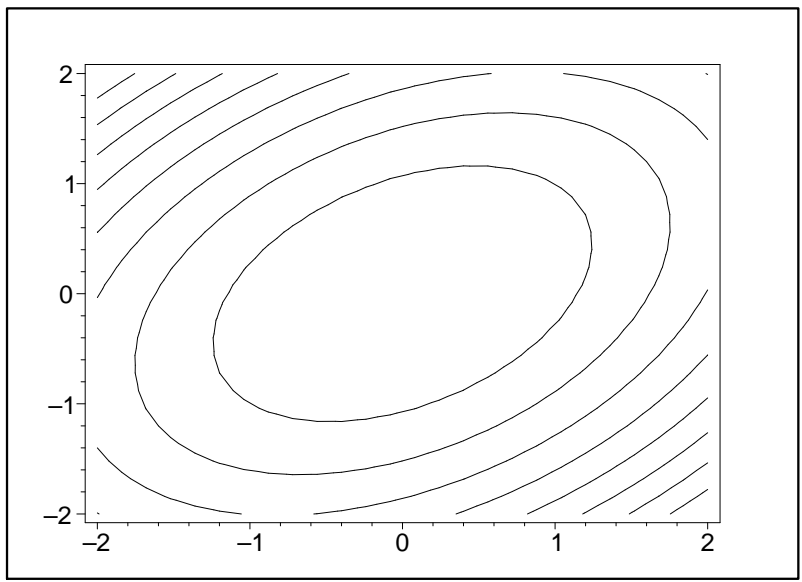}}
\caption{ Level sets of the norm $\|x\|_B$ used in example \ref{Ex2}. }
\end{center}
\end{figure}

\begin{figure}\label{figex2b}
\begin{center}
\includegraphics[width=2.9in]{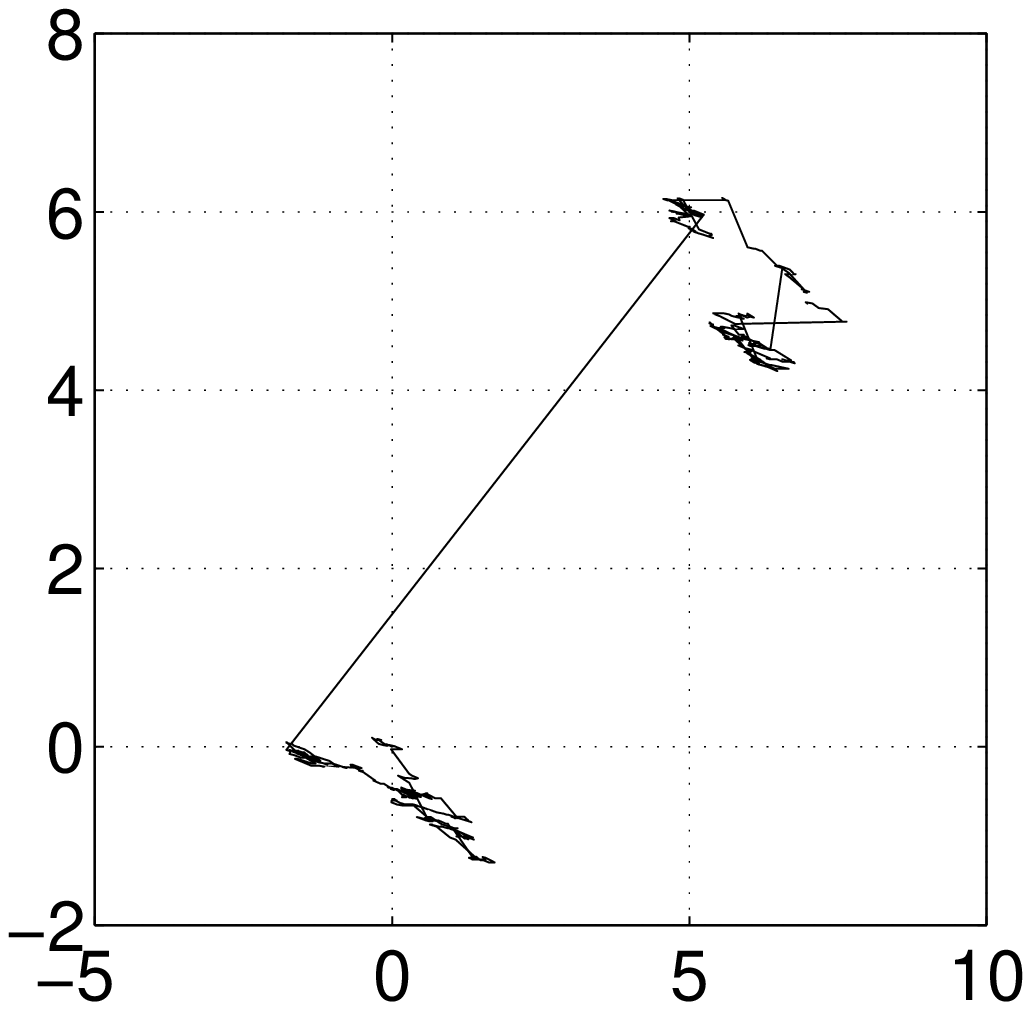}
\includegraphics[width=2.9in]{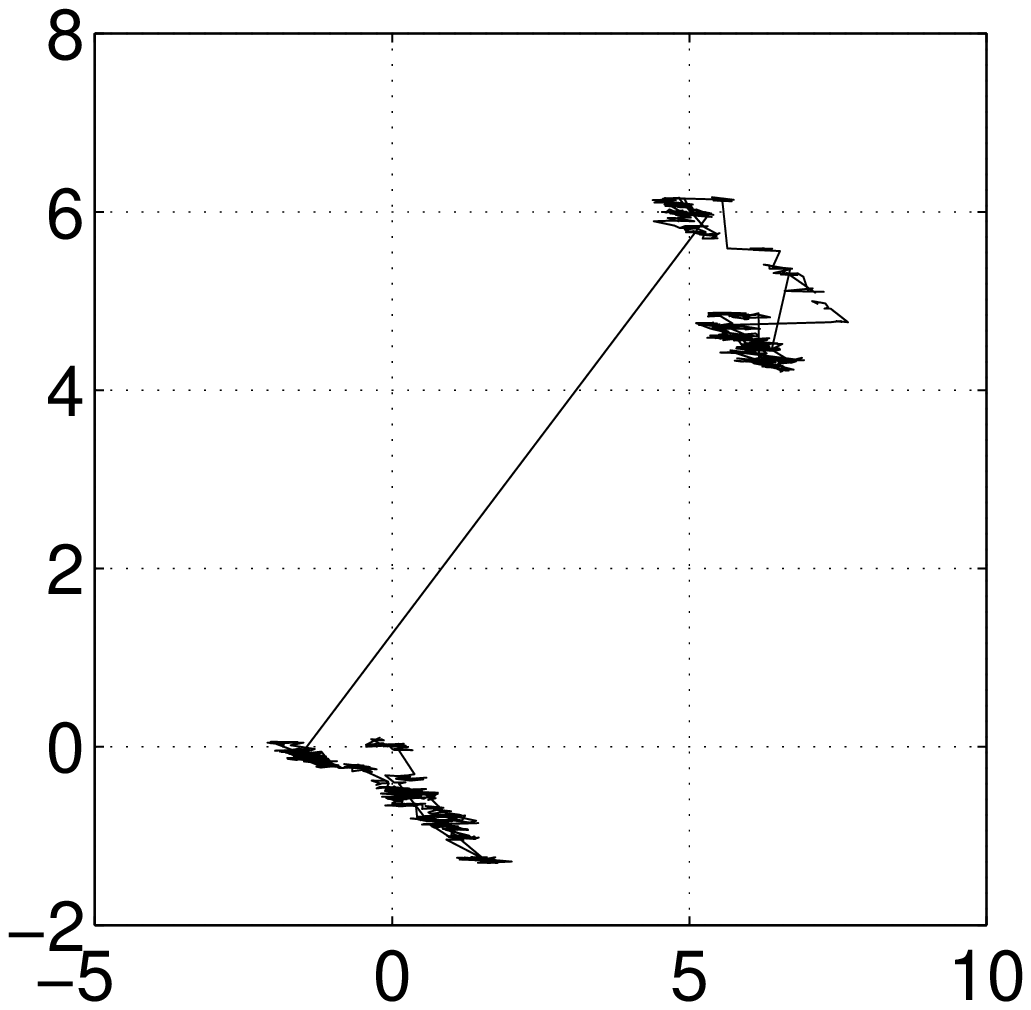}
\includegraphics[width=2.9in]{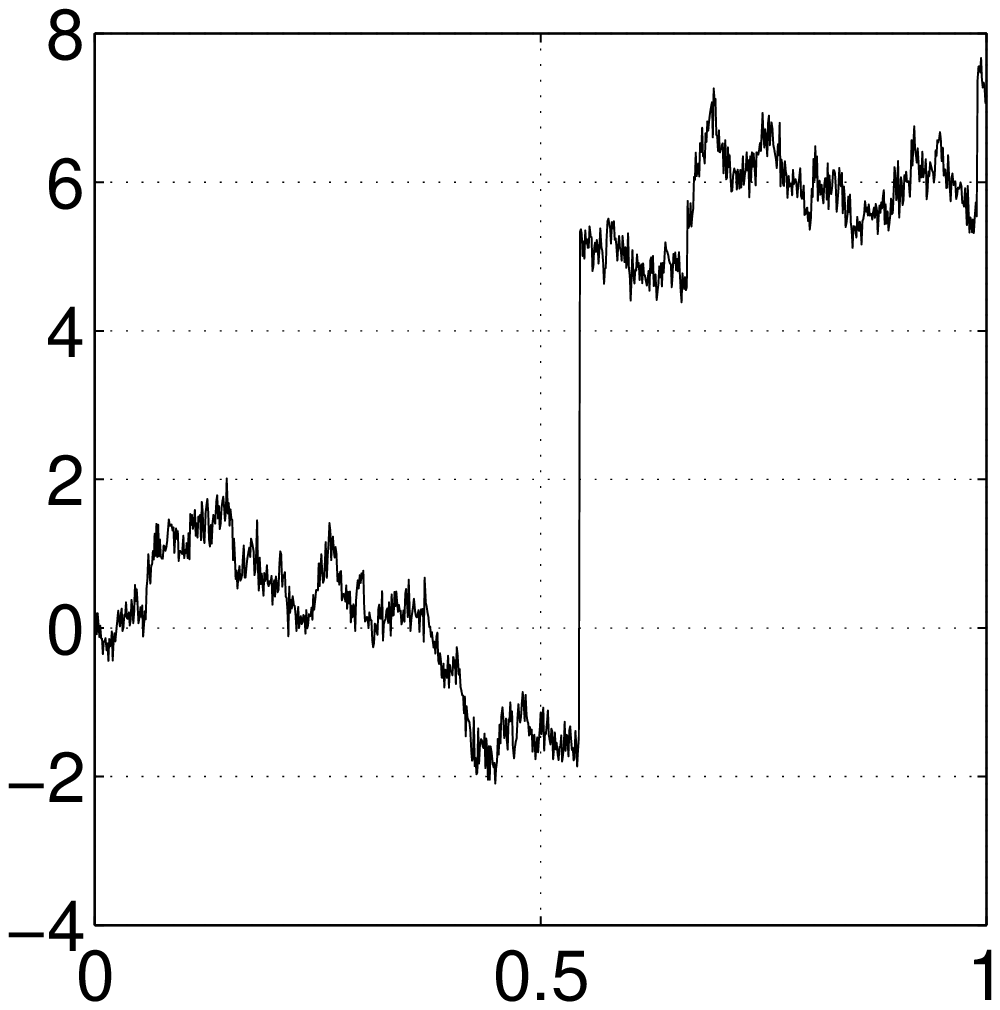}
\includegraphics[width=2.9in]{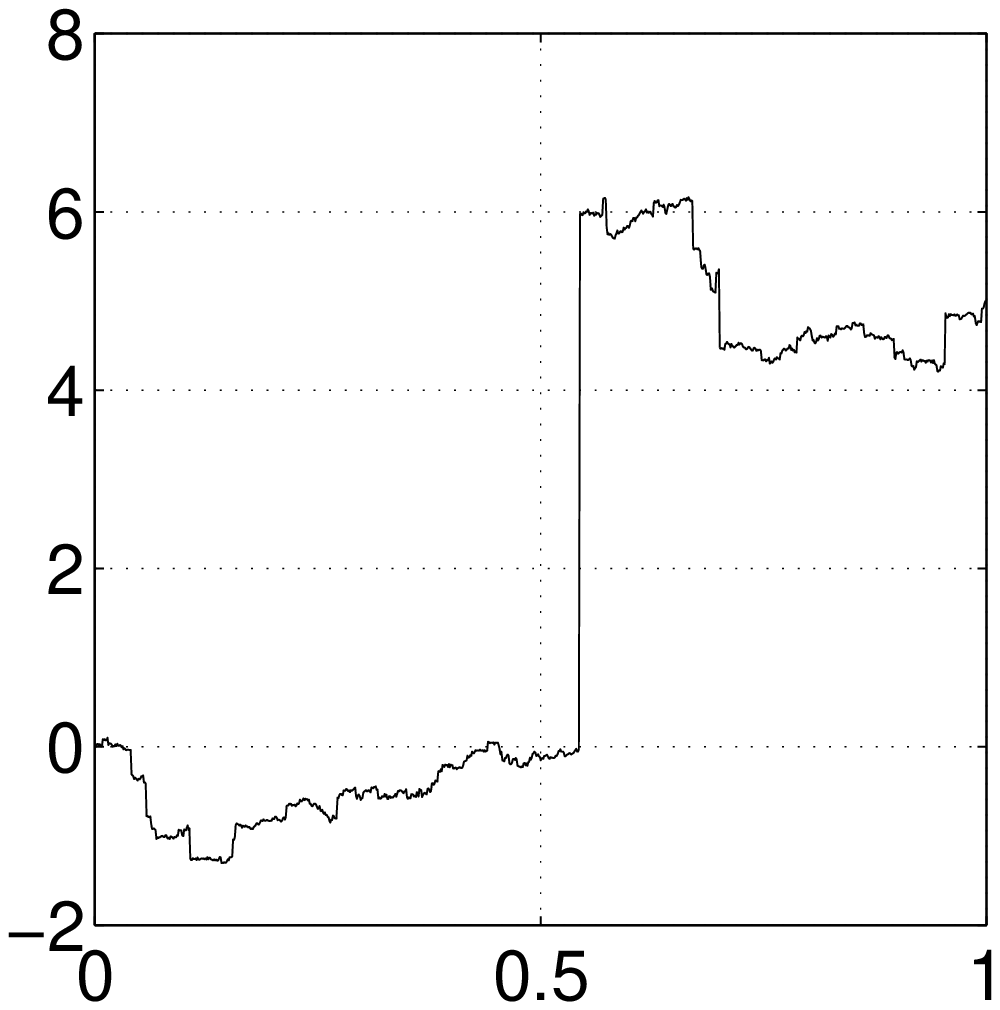}
\caption{Operator stable sample path for example \ref{Ex2}, whose exponent $B$ is not given in Jordan form.}
\end{center}
\end{figure}

\begin{example}\label{Ex2} This example is an operator stable process whose exponent
\[B=\left[\begin{array}{cc} 1/1.8 & 1/2 \\0 & 1/1.5\end{array}\right] .\]
The matrix $B$ has eigenvalue-eigenvector pairs $Bv_i=b_iv_i$ with $b_1=1/1.8$, $v_1=e_1$, $b_2=1/1.5$, and $v_2=(9/2)e_1+e_2$. As in Example \ref{ExB} we compute
\[t^B=\left[\begin{array}{cc} t^{1/1.8} & -(9/2)t^{1/1.8} +(9/2) t^{1/1.5} \\0 & t^{1/1.5}\end{array}\right]\]
We choose the norm \eqref{JMnorm} with $p=2$.  Compute $\|x\|_B^2=(9/10)x_1^2-(81/110)x_1 x_2+(903/880)x_2^2$ so that the unit sphere $S_B$ is an ellipse, whose major axis is rotated approximately $50^\circ$ counterclockwise from the $e_1$ direction.
Figure 9 
shows level sets of this norm.  The spectral measure $\lambda$ places equal masses of $1/4$ at each point where the unit sphere $S_B$ intersects the coordinate axes: $\pm c_i e_i$ where $c_1^2=10/9$ and $c_2^2=880/903$.
Here $\Lambda={\rm diag}(5/9,440/903)$, and
\[A_\epsilon=  \left[ \begin {array}{cc}  5.209&- 0.266\\\noalign{\medskip}- 0.266& 0.274\end {array} \right] .\]
The second coordinate $X_2(t)$ is symmetric stable with index $\alpha_2=1.5$, and the projection onto the remaining eigenvector $X_1(t)-(9/2)X_2(t)$ is stable with index $\alpha_1=1.8$.   These two stable marginals of $X(t)$ are not independent, since the spectral measure is not concentrated on the eigenvector axes.
Figure 10 
shows a typical sample path for this process.  The large jumps of the process are all of the form $t^{-B}v$ where $v=\pm c_1 e_1$ or $v=\pm c_2 e_2$, since we have concentrated the spectral measure at these points.  If $v=\pm c_1 e_1$ then, since $e_1$ is an eigenvector of $B$ (and hence of $t^B$), these jumps will be in the horizontal.  The remaining jumps lie along the orbits $\{\pm t^B c_2 e_2:t>0\}$.
\end{example}

\bigskip

\section{Exponents and Symmetries in two dimensions}\label{sec:sym-oss}

Section \ref{sec:sim} illustrates the wide range of possibilities represented by operator stable L\'evy processes in $\R^2$.  In this section we will provide a classification of such processes, according to the type of exponent, and the symmetry group.
Let $\X=\{X(t)\}_{t\ge 0}$ be an operator stable L\'evy process with exponent $B$,
and $\mu=\mathcal{L}(X(1))$,
as in Section~\ref{sec:os}.   Since the real parts of the eigenvalues of $B$ are greater than 1/2, after a change of coordinates, if needed, the exponent assumes one the following Jordan forms
\begin{equation}\label{Bi}
B_0= b I ,  \quad
B_1=\left[\begin{array}{cc} b_1 & 0 \\0 & b_2\end{array}\right],  \quad B_2=\left[\begin{array}{cc}b & -c \\c & b\end{array}\right],  \quad
B_3=\left[\begin{array}{cc}b & 0 \\1 & b\end{array}\right]
\end{equation}
where $b, b_1, b_2 > 1/2$, $b_1 \ne b_2$, and $c\ne 0$.
If $B=B_0$, then $\X$ is a multivariable stable process with index $\alpha=1/b$, and all maximal compact subgroups of $\mathrm{GL}(\R^2)$ are admissible as $\mathcal{S}(\mu)$. A genuine operator stable L\'evy process is obtained when $B=B_i$, $i=1,2,3$. Our first question is, what are possible symmetry groups?

To deal with this question, we need to review some basic facts about subgroups of the orthogonal group
$\mathcal{O}_2$ on $\R^2$, which can be found, e.g, in \cite{Barker07}. Recall that $\mathcal{O}_2$ consists of rotations and reflections,
\[
\mathcal{O}_2=\{R_{\theta},F_{\theta}:\theta\in[0,2\pi)\},
\]
 where
 \[
R_{\theta}=\left[\begin{array}{cc}
\cos\theta & -\sin\theta\\
\sin\theta & \cos\theta\end{array}\right]\quad\mbox{}{\mathrm{and}}\quad F_{\theta}=\left[\begin{array}{cc}
\cos\theta & \sin\theta\\
\sin\theta & -\cos\theta\end{array}\right].
\]
 $R_{\theta}$ is a rotation counter-clockwise by $\theta$ and $F_{\theta}$
is a reflection through the line of angle $\theta/2$ passing through the
origin. The following rules of composition hold: $R_{\theta_{1}}R_{\theta_{2}}=R_{\theta_{1}+\theta_{2}}$,
$F_{\theta_{1}}F_{\theta_{2}}=R_{\theta_{1}-\theta_{2}}$, $R_{\theta_{1}}F_{\theta_{2}}=F_{\theta_{1}+\theta_{2}}$,
$F_{\theta_{2}}R_{\theta_{1}}=F_{\theta_{2}-\theta_{1}}$.

The group of rotations $\mathcal{O}_2^{+}=\{R_{\theta}:\theta\in[0,2\pi)\}$ is the only infinite proper compact subgroup of $\mathcal{O}_2$.
There are also only two kinds of finite subgroups
of $\mathcal{O}_2$ (modulo the orthogonal conjugacy, see \cite[Ch. VII.3]{Barker07}):
\begin{enumerate}
\item Cyclic group $\mathcal{C}_{n}=\{R_{k2\pi/n}:k=0,\dots,n-1\}$, \ $n\ge1$,
\item Dihedral group $\mathcal{D}_{n}=\{R_{k2\pi/n},F_{k2\pi/n}:k=0,\dots,n-1\}$,
\ $n\ge 1$.
\end{enumerate}
Notice that $\mathcal{C}_1= \{I\}$, $\mathcal{C}_2= \{I, -I\}$, $\mathcal{D}_1= \{I, F_0\}$, and $\mathcal{D}_2= \{I, F_0, -I, -F_0\}$, where
$$
F_{0}=\left[\begin{array}{cc}
1 & 0\\
0 & -1\end{array}\right]
$$
is the reflection with respect to the $x$-axis. We will also need $\mathcal{D}_1^*= \{I, -F_0\}$,  the group of reflection with respect to the $y$-axis, which is orthogonally conjugate to $\mathcal{D}_1$.

The next result characterizes the possible symmetries of the distribution of $X(t)$ in the truly operator stable case where $B=B_i$ in \eqref{Bi} for some $i=1,2,3$.  In view of \eqref{eq:def-op-stable-mu}, the symmetry group do not depend on $t$.  Remarkably, once the exponent takes the Jordan form, all symmetries must be orthogonal, not just conjugate to an orthogonal matrix.

\begin{theorem}\label{os2}
	Let $\X=\{X(t)\}_{t\ge 0}$ be a full operator stable L\'evy processes on $\R^2$ with an exponent $B$ in the Jordan form \eqref{Bi}, and let $\mu=\mathcal{L}(X(1))$. Then the following hold.
	\begin{itemize}
	\item[(i)]  If $B=B_1$, then $\mathcal{S}(\mu)$ is either $\mathcal{C}_1$, $\mathcal{C}_2$, $\mathcal{D}_1$,  $\mathcal{D}_1^*$, or
	$\mathcal{D}_2$.
	\item[(ii)]  If $B=B_2$, then $\mathcal{S}(\mu)$ is either $\mathcal{C}_n$, $n\ge 1$, or $\mathcal{O}_2$.
	\item[(iii)]  If $B=B_3$, then $\mathcal{S}(\mu)$ is either $\mathcal{C}_1$ or $\mathcal{C}_2$.
	\end{itemize}
	\end{theorem}
\begin{proof}
Suppose that $\mu$ has an exponent $B=B_i$, $i = 1, 2, 3$,  and let  $B_c$ be a commuting exponent, see Section \ref{sec:os}. If $\mathcal{S}(\mu)$ is finite, then $B_i=B_c$,  otherwise $B_c$ can be different from $B_i$.  The symmetries $\mathcal{S}(\mu)$ defined in \eqref{sg-mu} form a compact subgroup of the centralizer $C(B_c)$,
\begin{equation}\label{centr}
	\Sa(\mu) \subset C(B_c) := \{A \in \mathrm{GL}(\R^2): AB_c=B_cA\}.
\end{equation}
First consider finite symmetry groups $\Sa(\mu)$, so that $B_c=B_i$. If $i=1$,
$$
C(B_1) =  \left\{ \left[\begin{array}{cc} \alpha & 0 \\0 & \beta \end{array}\right] : \alpha \beta \ne 0 \right\},
$$
and since $\Sa(\mu)$ is finite (and thus compact),
$$
\Sa(\mu) \subset  \left\{ \left[\begin{array}{cc} \alpha & 0 \\0 & \beta \end{array}\right] : |\alpha|= |\beta|=1 \right\}.
$$
Thus $\Sa(\mu)$ is either $\mathcal{C}_1$, $\mathcal{C}_2$, $\mathcal{D}_1$,  $\mathcal{D}_1^*$, or $\mathcal{D}_2$, as claimed. If $i=2$, then
$$
C(B_2) = \left\{ \left[\begin{array}{cc} \alpha & -\beta \\\beta & \alpha \end{array}\right] : \alpha^2+ \beta^2 > 0 \right\}.
$$
Since $\Sa(\mu)$ is finite (and thus compact),
$$
\Sa(\mu) \subset \left\{ \left[\begin{array}{cc} \alpha & -\beta \\\beta & \alpha \end{array}\right] : \alpha^2+ \beta^2 = 1 \right\},
$$
so that $\Sa(\mu) = \mathcal{C}_n$,  for some  $n \ge 1$.
If $i=3$,
$$
C(B_3) =  \left\{ \left[\begin{array}{cc} \alpha & 0 \\ \beta & \alpha \end{array}\right] : \alpha \ne 0, \beta \in \R \right\},
$$
and since $\Sa(\mu)$ is finite,
$$
\Sa(\mu) \subset  \left\{ \left[\begin{array}{cc} \alpha & 0 \\0 & \alpha \end{array}\right] : |\alpha|= 1 \right\}.
$$
Thus $\Sa(\mu)$ is either $\mathcal{C}_1$ or $\mathcal{C}_2$, as claimed.

Now we consider infinite symmetry groups $\Sa(\mu)$, so that $\Sa(\mu)=W^{-1}\mathcal{O}_2W$ for some symmetric positive definite matrix $W$, see \eqref{sg1}. From \eqref{centr}, $WB_c W^{-1}$ commutes with every orthogonal transformation. Thus $WB_c W^{-1}$ is a multiple of the identity matrix,
which yields
\begin{equation}\label{Bc}
B_c  = \beta I
\end{equation}
Since $T\mathcal{O}_2= \mathcal{Q}_2$,
$$
B_i= B_c + W^{-1} K W =  W^{-1}(\beta I +K) W
$$
for some skew symmetric matrix $K$, and so
$$
B_i= \gamma W^{-1}R_{\phi} W
$$
for some $\gamma \ne 0$ and $\phi \in [0, 2\pi)$.
This equation eliminates the cases $i=1$ and $i=3$ by comparing the eigenvalues on the left and right hand side.  Thus $i=2$ and $B_2= \alpha R_\psi$ for some $\psi \in (0, \pi) \cup (\pi, 2\pi)$, from which we have
$$
\alpha R_\psi = B_2 = \gamma W^{-1}R_{\phi} W.
$$
Comparing the determinants of both sides gives $\alpha= \gamma$. Hence
$$
R_\psi = W^{-1}R_{\phi} W.
$$
Since the sets of eigenvalues of both sides of this equation must be the same, $\phi=\psi$ or $\phi= 2\pi-\psi$. If $\phi=\psi$ then $W R_{\psi}= R_{\psi}W$ for $\psi \in (0, \pi) \cup (\pi, 2\pi)$. A direct verification of this equation reveals that $W = \kappa R_{\tau}$ is a multiple of a rotation.  (In fact, $W$ is a scalar multiple of the identity, since it is also symmetric and positive definite.)  Therefore,
$$
\Sa(\mu)=  (\kappa R_{\tau})^{-1} \mathcal{O}_2 \kappa R_{\tau} = \mathcal{O}_2,
$$
as claimed. If $\phi= 2\pi-\psi$, then
$$
R_\psi = W^{-1}R_{2\pi - \psi} W = W^{-1}F_0 F_{\psi}W= W^{-1}F_0 R_{\psi}F_0W
$$
or
$$
(F_0 W)R_\psi =R_{\psi}(F_0W).
$$
By the same reason as above, one can verify that $F_0 W= \kappa R_{\tau}$ is a multiple of rotation.
Hence $W= \kappa F_{-\tau}$ and
$$
\Sa(\mu)=  (\kappa F_{-\tau})^{-1} \mathcal{O}_2 \kappa F_{-\tau} = \mathcal{O}_2.
$$
This proves that $B=B_2$ and $\Sa(\mu)= \mathcal{O}_2$ provided $\Sa(\mu)= W^{-1}\mathcal{O}_2 W$.
\end{proof}

\begin{remark}
Any operator stable law can be transformed to one in which the exponent takes the Jordan form \eqref{Bi} by a simple change of basis.  Theorem \ref{os2} shows that the Jordan basis renders all symmetries orthogonal, and then the Jordan form determines which symmetries are possible.  
\end{remark}

Operator stable laws are parameterized by their exponents and spectral measures. Therefore, it is useful to have their symmetries described in terms of these parameters.  Recall that $\mathcal{O}^{+}_2$ denotes the subgroup of rotations of the orthogonal group  $\mathcal{O}_2$.

\begin{theorem}\label{MMMquery}
	Let $\X=\{X(t)\}_{t\ge 0}$ be a full operator stable L\'evy process in $\R^2$ with exponent $B$ and no Gaussian component, and let $\mu=\mathcal{L}(X(1))$. Suppose that $B$ is given in the Jordan form \eqref{Bi} and that the spectral measure $\lambda$ is determined by the polar decomposition \eqref{eq:sm} relative to $S_B=S^1$, the Euclidean unit sphere of $\R^2$.  Let $\mathcal{S}_0(\lambda)=\{A\in \mathrm{GL}(\R^{2}): A\lambda=\lambda\}$ denote the strict symmetry group of the spectral measure.
	\begin{itemize}
	\item[(a)] If $B=B_1$, then $\Sa(\mu)=\mathcal{S}_0(\lambda) \cap \mathcal{D}_2$.
	\item[(b)] If $B=B_2$, then either $\Sa(\mu)=\mathcal{S}_0(\lambda) \cap \mathcal{O}_2^{+} = \mathcal{C}_n$ for some $n \ge 1$, or $\Sa(\mu)=\mathcal{S}_0(\lambda) = \mathcal{O}_2$.  
	\item[(c)] If $B=B_3$, then $	\Sa(\mu)=\mathcal{S}_0(\lambda) \cap \mathcal{C}_2$ .
\end{itemize}
\end{theorem}

\begin{proof}
Let $\nu$ be the L\'evy measure of $\mu$. Since $\mu$ does not have a Gaussian part, we have
\begin{equation}\label{sss}
\Sa(\mu)= \mathcal{S}_0(\nu)=\{A\in \mathrm{GL}(\R^{2}): A\nu=\nu\}	
\end{equation}
as in \eqref{sm=sn}.  First we will show that if $B=B_i$, $i=1,2,3$ and $\Sa(\mu)$ is finite, then
\begin{equation}\label{S2}
	\Sa(\mu)=\mathcal{S}_0(\lambda) \cap \{A\in \mathcal{O}_2: AB=BA\}.
\end{equation}
Indeed, recall \eqref{eq:Lm-ost} and \eqref{eq:sm} for $S_B=S^1$:
\begin{equation*}\begin{split}
\nu(E)&=\int_{S^1}\int_{0}^{\infty}\mathbf{1}_{E}(s^{B}u)s^{-2}\, ds\lambda(du) ,
 \qquad \text{$E \in \mathcal{B}(\mathbb{R}^{2})$, where}\\
\lambda(F)&=\nu(\{x: x=t^Bu, \ \text{for some } (t,u) \in [1,\infty)\times F \}),  \quad F \in \mathcal{B}(S^1).
\end{split}\end{equation*}
Let $A \in \Sa(\mu)$,  $\Sa(\mu)$ being finite. Then $A\in \mathcal{O}_2$ by Theorem \ref{os2} and $A$ commutes with $B$. For every $F \in \mathcal{B}(S^1)$, $A^{-1}F \in \mathcal{B}(S^1)$ and
\begin{align*}
	 \lambda(A^{-1}F) &= \nu(\{x: x=t^BA^{-1}v, \ \text{for some } (t,v) \in [1,\infty)\times F \})\\
	  &= \nu(A^{-1}\{x: x=t^Bv, \ \text{for some } (t,v) \in [1,\infty)\times F \}) = \lambda(F)
\end{align*}
because $\Sa(\mu)=\mathcal{S}_0(\nu)$ from \eqref{sss}. Hence $A \in \mathcal{S}_0(\lambda)$.  The proof of the opposite inclusion in \eqref{S2} uses similar arguments and is omitted.

{\it Proof of (a).} A direct verification shows that $B_1$ commutes with $\mathcal{D}_2$.
Thus by \eqref{S2}
$$
\mathcal{S}_0(\lambda) \supset \Sa(\mu) \supset \mathcal{S}_0(\lambda) \cap \mathcal{D}_2.
$$
Since $\Sa(\mu) \subset \mathcal{D}_2$ by Theorem \ref{os2}, we get (a).
\medskip

{\it Proof of (b).} By Theorem \ref{os2} $\Sa(\mu)=\mathcal{C}_n$ for some $n \ge 1$, or $\Sa(\mu)=\mathcal{O}_2$. Suppose that $\Sa(\mu)=\mathcal{C}_n$. Since $\mathcal{O}_2^+$ commutes with $B_2$, by \eqref{S2} we have
$$
\mathcal{S}_0(\lambda) \supset \Sa(\mu) \supset \mathcal{S}_0(\lambda) \cap \mathcal{O}_2^+.
$$
Thus $\Sa(\mu) = \mathcal{S}_0(\lambda) \cap \mathcal{O}_2^+ =\mathcal{C}_n$.

Suppose $\Sa(\mu)=\mathcal{O}_2$. Then $R_{\theta} \in \mathcal{S}_0(\nu)$ for every $\theta$ by \eqref{sss}. Since $R_{\theta}$ commutes with $B_2$, $R_{\theta} \in \mathcal{S}_0(\lambda)$ by the same line of arguments as in the proof of \eqref{S2}. Hence $\mathcal{S}_0(\lambda) \supset \mathcal{O}^+_2$, which implies that $\lambda$ is a finite full measure in $\R^2$. Then $\lambda$ is a constant multiple of a probability measure, so $\mathcal{S}_0(\lambda)$ must be maximal by \cite[Theorem 2]{Meerschaert95}, and hence $\mathcal{S}_0(\lambda)= \mathcal{O}_2$.

\medskip

{\it Proof of (c).} It follows from \eqref{S2} because $\mathcal{C}_2$ obviously commutes with $B_3$.
\end{proof}

\begin{remark}\label{OSexists}
With the help of Theorem \ref{MMMquery}, it is possible to explicitly construct an operator stable process with any given exponent $B_i$ for $i=1,2,3$ in the Jordan form \eqref{Bi} and any admissible symmetry group.  For example, let  $\lambda$ be concentrated at four points $(\pm 2^{-1/2}, \pm 2^{-1/2})$. Choosing masses at these points appropriately, any subgroup of $\mathcal{D}_2$ is realized as  $\mathcal{S}_0(\lambda)$. By Theorem \ref{MMMquery}, all cases of $\Sa(\mu)$ are realized by this example  when $B=B_1$ and $B=B_3$. When $B=B_2$, we only get $\mathcal{C}_1$ and $\mathcal{C}_2$. To get $\mathcal{S}(\mu)=\mathcal{C}_n$, $n\ge 3$, we take $\lambda$ concentrated at vertices of a regular $n$-gon inscribed into the unit circle with one vertex at $(1,0)$ and equal masses at all the vertices. Then $\mathcal{S}_0(\lambda)= \mathcal{D}_n$, so by Theorem~\ref{MMMquery}, $\Sa(\mu)= \mathcal{D}_n \cap \mathcal{O}_2^+ = \mathcal{C}_n$.
$\Sa(\mu)=\mathcal{O}_2$ when $B=B_2$ and $\lambda$ is a uniform measure on $S^1$.
\end{remark}

%

\begin{remark}\label{orbits}
It is interesting to see how much an exponent affects the symmetry. Consider a measure $\lambda$ with $\mathcal{S}_0(\lambda)= \mathcal{D}_n$ described in Remark \ref{OSexists}, with $n=3$. Then, by Theorem \ref{MMMquery}, $\Sa(\mu)=\mathcal{D}_1$ when $B=B_1$, $\Sa(\mu)=\mathcal{C}_3$ when $B=B_2$, and $\Sa(\mu)=\mathcal{C}_1$ when $B=B_3$.  Figure 11
illustrates the diagonal case $B=B_1$, in which the L\'evy measure $\nu$ in \eqref{eq:Lm-ost} is symmetric with respect to reflection about the vertical axis.  Here we take $b_1=1/1.8$ and $b_2=1/1.5$, but any case with $b_1\neq b_2$ appears similar.
Figure 12
illustrates the complex case $B=B_2$, where the L\'evy measure is symmetric with respect to rotations that are a multiple of $2\pi/3$.  Figure 13
illustrates the nilpotent case $B=B_3$, and here the L\'evy measure has no nontrivial symmetries.  All three cases have the same spectral measure, but a different exponent.   Hence the spectral measure and the exponent are both important in determining the symmetries.
\end{remark}

\begin{figure}\label{figB1orbits}
\begin{center}
\rotatebox{270}{\includegraphics[width=2.9in,height=3.2in]{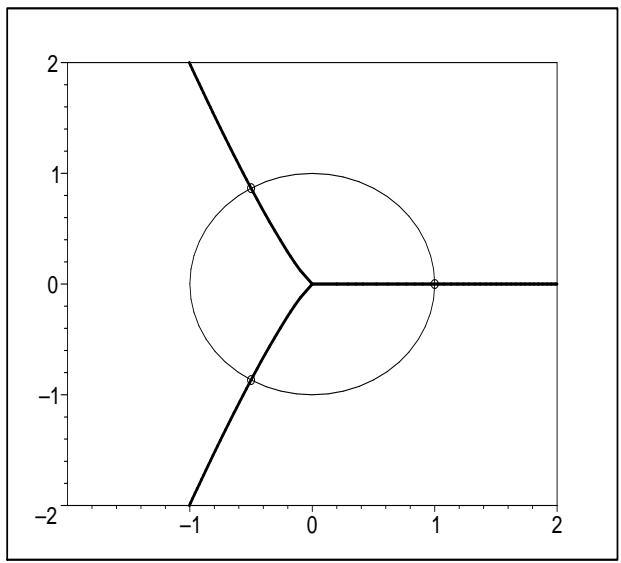}}
\caption{Support of the L\'evy measure (thick lines) for Remark \ref{orbits}, showing the effect of the exponent $B_1$ on the symmetry group.  The spectral measure gives equal weight to three equally spaced points on the unit circle, so that $\mathcal{S}_0(\lambda)= \mathcal{D}_3$.  In this case $B=B_1$, we have $\Sa(\mu)=\mathcal{D}_1$.}
\end{center}
\end{figure}

\begin{figure}\label{figB2orbits}
\begin{center}
\rotatebox{270}{\includegraphics[width=2.9in,height=3.2in]{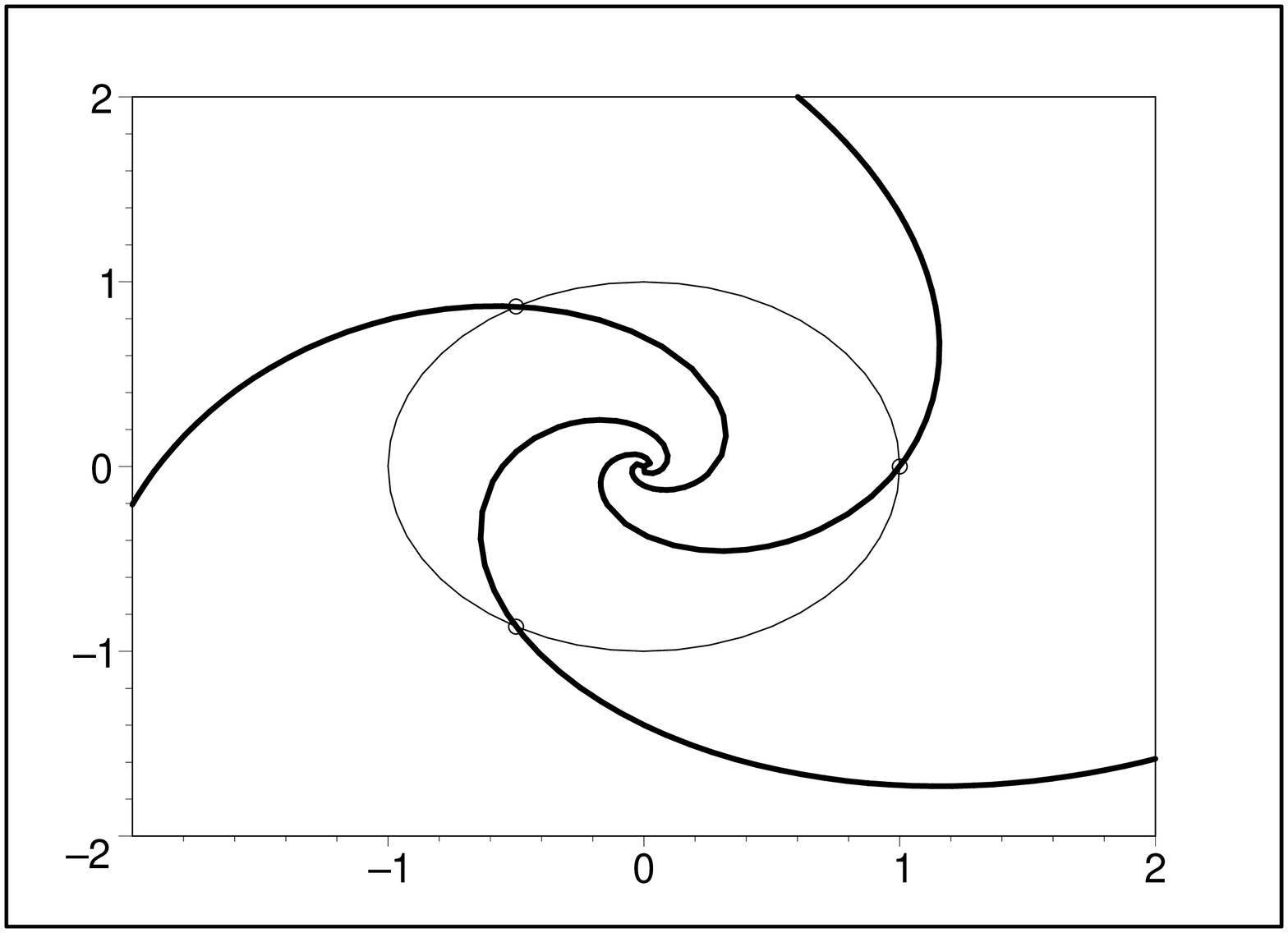}}
\caption{Support of the L\'evy measure for Remark \ref{orbits}, showing the effect of the exponent $B_2$ on the symmetry group.  Here $\Sa(\mu)=\mathcal{C}_3$.}
\end{center}
\end{figure}

\begin{figure}\label{figB3orbits}
\begin{center}
\rotatebox{270}{\includegraphics[width=2.9in,height=3.2in]{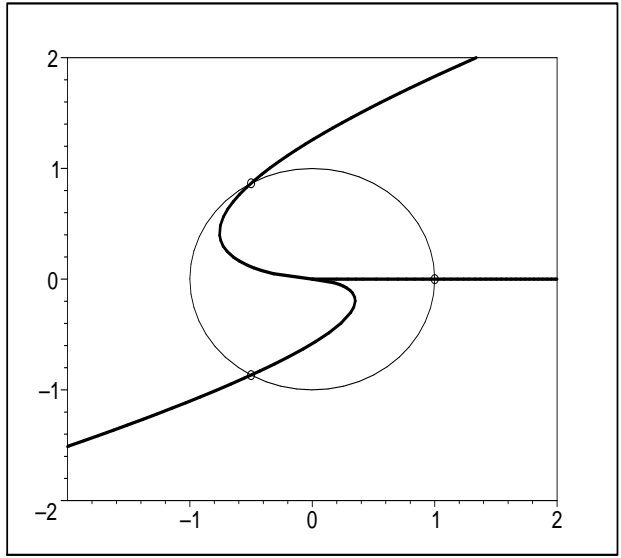}}
\caption{Support of the L\'evy measure for Remark \ref{orbits}, showing the effect of the exponent $B_3$ on the symmetry group.  Here $\Sa(\mu)=\mathcal{C}_1$.}
\end{center}
\end{figure}

\begin{remark}\label{concrete}
In order to tie the theoretical results of this section back to the concrete examples in Section \ref{sec:sim}, we compute the symmetry group $\Sa(\mu)$ for a few interesting cases.
For Example \ref{ExA} we have $\Sa(\mu)=\mathcal{S}_0(\lambda)=\mathcal{D}_2$ by Theorem \ref{MMMquery} (a), since the exponent $B=B_1$ in \eqref{Bi}, and spectral measure $\lambda$ gives equal mass to the four points $\pm e_1,\pm e_2$.
The spectral measure in Example \ref{ExA1} is uniform on the unit sphere, so that $\mathcal{S}_0(\lambda)=\mathcal{O}_2$, but the symmetry is of the form $B=B_1$ in \eqref{Bi}, so the symmetry group $\Sa(\mu)=\mathcal{D}_2$ by Theorem \ref{MMMquery} (a).
The construction in Example \ref{ExN} yields $\mathcal{S}_0(\lambda)=\mathcal{D}_2$.  Then $\Sa(\mu)=\mathcal{C}_2$ since the exponent $B=B_3$ is nilpotent, by Theorem \ref{MMMquery} (b).  In Example \ref{ExX} we also have $\mathcal{S}_0(\lambda)=\mathcal{D}_2$, and then $\Sa(\mu)=\mathcal{C}_2$ by Theorem \ref{MMMquery} (c).  Example \ref{ExC} has $\Sa(\mu)=\mathcal{D}_1$ since the spectral measure is symmetric with respect to reflection across the $e_1$-axis: $B=B_1$, and $F_0\lambda=\lambda$, but $-I\lambda \neq \lambda$.
\end{remark}

\section{Operator self-similar processes}\label{xxx}

In this section, we discuss more general operator self-similar processes, whose increments need not be independent or stationary.  From now on, assume that the operator self-similar process $\X$ is proper (i.e., for every $t>0$ the smallest hyperplane supporting the distribution of $X(t)$ equals $\R^{d}$), stochastically continuous, and $X(0)=0$.  Under these assumptions, the real parts of eigenvalues of the exponent $B$ are positive \cite[Theorem 4]{Hudson82}.  Denote by $\mathcal{S}(\X)$ the set of linear operators $A$ in $\mathrm{GL}(\R^{d})$ such that
\begin{equation}
\{AX(t)\}_{t\ge0}\,\eqfd\,\{X(t)\}_{t\ge0}. \label{sg}
\end{equation}
The symmetries of $\X$ form a compact subgroup of $\mathrm{GL}(\R^{d})$ as long as $\X$ is proper.  The symmetry group can be seen as minimal information about a
multidimensional stochastic process.
Hudson and Mason \cite[Theorem 2]{Hudson82} proved that
\begin{equation}
\mathcal{E}(\X)=B+T\mathcal{S}(\X),\label{exposs}
\end{equation}
 where $B\in\mathcal{E}(\X)$ is arbitrary and $T\mathcal{S}(\X)$
is the tangent space of $\mathcal{S}(\X)$ at the identity. Maejima \cite{Maejima98} showed that one can always find a {commuting exponent} $B_c \in \mathcal{E}(\X)$ such that $AB_c=B_cA$ for all $A\in \mathcal{S}(\X)$.

A shift is included in the symmetry group $\Sa(\mu)$ defined in \eqref{sg-mu} for operator stable L\'evy processes, since the definition \eqref{eq:def-op-stable-mu} also includes a shift.  For operator self-similar processes, the definition \eqref{eq:ss} does not include a shift, so it is natural that the definition \eqref{sg} for the symmetry group $\mathcal{S}(\X)$ of an operator self-similar process does not allow a shift.  The following lemma connects $\mathcal{S}(\X)$ with $\Sa(\mu)$ in the operator stable case.

\begin{lemma}\label{}
Let $\X=\{X(t)\}_{t\ge 0}$ be a strictly operator stable L\'evy process with exponent $B$. Suppose that 1 is not an eigenvalue of $B$. Then $\mathcal{S}(\X)= \Sa(\mu)$, where $\mu=\mathcal{L}(X(1))$.
\end{lemma}

\begin{proof}
Since $\{X(t)\}\eqfd\{AX(t)\}$ if and only if $X(1)\dlim AX(1)$, we have $\mathcal{S}(\X)=\mathcal{S}_0(\mu)$, so it suffices to show that $\Sa(\mu) =\mathcal{S}_0(\mu)$ (see definitions \eqref{sg-mu} and \eqref{ssg-mu}). Let $A \in \Sa(\mu)$, so that $AX(1)$ and $X(1) - b$ are identically distributed for some $b\in\R^d$. Since the real parts of eigenvalues of all exponents of $\mu$ are the same (see \cite[Corollary 7.2.12]{Meerschaert01}), we may take $B$ as a commuting exponent. Then, for every $t>0$ we have
\begin{align*}
AX(t) & \dlim X(t) - tb \dlim t^BX(1) - tb = t^B(AX(1) + b) - tb \\
& = At^BX(1) + t^Bb - tb \dlim AX(t) + t^Bb - tb.
\end{align*}
Thus $(t^B-t)b=0$ for all $t>0$, and since 1 is not an eigenvalue of $B$, $b=0$.
Hence $A \in \mathcal{S}_0(\mu)$. The converse inclusion, $\Sa(\mu) \supset\mathcal{S}_0(\mu)$, is obvious.
\end{proof}

\begin{remark}\label{osvsoss}
Full dimensional operator stable L\'evy processes, and proper operator self-similar processes, form two distinct classes.  Neither class is contained in the other.  Take $Z(t)$ a spherically symmetric L\'evy process on $\R^d$ whose marginals are Cauchy.  Then $b+Z(t)$ is an operator stable L\'evy process, but it is not operator self-similar.  The process $X(t)=Z(t^p)$ for $p>1$ is operator self-similar but not L\'evy.  Remark \ref{fcp} provides examples of operator self-similar processes for which none of the one-dimensional distributions are operator stable.   The process $X(t)=vt+Z(t)$ is a strictly operator stable L\'evy process and also a proper operator self-similar process.  If we take $\mu=\mathcal{L}(X(1))$ then $\Sa(\mu)={\mathcal O}_d$ but $\mathcal{S}_0(\mu)=\mathcal{S}(\X)$ consists of the orthogonal transformations that fix the vector $v$.
\end{remark}

Theorem \ref{os2} and Remark \ref{OSexists} showed how to construct an operator stable L\'evy process with any admissible symmetry group.  The group $\mathcal{O}_2^+$ (and groups conjugated to it) were excluded, since they are not maximal (see Section \ref{sec:oss}).  This raises a question, is it possible to have $\mathcal{S}(\X)=\mathcal{O}_2^+$ for some operator self-similar (not L\'evy) processes? The answer is affirmative, as shown in the following example.

\begin{example}\label{O+}
Consider a complex valued process
$$
X(t) = t^{\beta} \exp \left( i(\Theta + \log t) \right),  \quad t>0,
$$
where $\beta>0$, $\Theta$ is a uniform random variable on $[0,2\pi]$ and  $X(0)=0$. Since for any $\phi \in \R$
$$
\{e^{i\phi}X(t)\}_{t\ge 0} \eqfd \{X(t)\}_{t\ge 0},
$$
$\X$ as a process in $\R^2$,
$$
X(t)= t^{\beta} \left[\begin{array}{c}\cos(\Theta + \log t) \\ \sin(\Theta + \log t)\end{array}\right]
$$
is a self-similar with index $\beta$ and $\mathcal{O}_2^+ \subset \mathcal{S}(\X)$.   By \eqref{exposs}, $I$ and $B_2$ are exponents of $\X$ ($B_2$ with $b=\beta$ and arbitrary $c$). If $A \in \mathcal{S}(\X)$ then
$$
AX(1) \dlim X(1)
$$
which implies $A \in \mathcal{O}_2$. Thus
$$
\mathcal{O}_2^+ \subset \mathcal{S}(\X) \subset \mathcal{O}_2.
$$
Consider the process $\{F_0 X(t)\}_{t\ge 0}$, where $F_0$ is the reflexion with respect to the $x$-axis,
\begin{align*}
F_0 X(t) = t^{\beta} \left[\begin{array}{c} \cos(\Theta + \log t)  \\ -\sin(\Theta + \log t) \end{array}\right].
\end{align*}
If $F_0 \in \mathcal{S}(\X)$, then for $t_1=1$ and $t_2=e^{\pi/2}$ we would have
\begin{align*}
(F_0X(1), F_0X(e^{\pi/2})) \dlim (X(1), X(e^{\pi/2})),
\end{align*}
or
\begin{align*}
 \left( \left[\begin{array}{c}\cos \Theta \\ -\sin \Theta\end{array}\right] , e^{\beta\pi/2}\left[\begin{array}{c}- \sin \Theta \\-\cos \Theta \end{array}\right]\right) \dlim
  \left( \left[\begin{array}{c}\cos \Theta \\ \sin \Theta\end{array}\right] , e^{\beta\pi/2}\left[\begin{array}{c}- \sin \Theta \\ \cos \Theta \end{array}\right]\right).
\end{align*}
This equality written in $\R^4$ means
$$
(\cos \Theta, -\sin \Theta, -\sin \Theta, -\cos \Theta) \dlim (\cos \Theta, \sin \Theta, -\sin \Theta, \cos \Theta),
$$
which is impossible since the sum of the first and the fourth random variables on the left hand side is $0$, while on the right hand side is $2\cos \Theta$. Hence $F_0 \notin \mathcal{S}(\X)$, which yields $\mathcal{S}(\X) = \mathcal{O}_2^+$.
\end{example}

\begin{remark}\label{diag-action}
Example \ref{O+} is consistent with the result that symmetry groups of probability measures must be maximal \cite[Theorem 2]{Meerschaert95}, even though $\mathcal{O}_2^+$ is not a maximal subgroup of $\mathrm{GL}(\R^{2})$.  This is because, for $A$ in $\mathcal{S}(\X)$, we not only require $AX(t)$ identically distributed with $X(t)$ for a single $t>0$, but also that $(AX(t_1), \ldots,AX(t_p))$ is identically distributed with $(X(t_1), \ldots,X(t_p))$  for all finite-dimensional distributions.  We say that $\mathcal{O}_2^+$ acts diagonally in this case, and we identify $A$ with corresponding element of $\mathrm{GL}(\R^{2p})$ defined by $(x_1,\ldots,x_p)\mapsto (Ax_1,\ldots,Ax_p)$ for $x_1,\ldots,x_p\in \R^2$.  In Example \ref{O+} the diagonal action of $\mathcal{O}_2^+$ is a maximal subgroup of $\mathrm{GL}(\R^{4})$, see the proof of Theorem 1 in \cite{Meerschaert95}.
\end{remark}

The exponents of a proper operator self-similar process are related to the symmetry group by \eqref{exposs}, there always exists a commuting exponent, and the eigenvalues of any exponent all have positive real part.  These were the crucial facts used in the proof of Theorem \ref{os2}. Hence we can also characterize the symmetry group of a proper operator self-similar process in $\R^2$ in terms of the exponent $B$ in Jordan form.  The proof is identical to Theorem \ref{os2}, except that here we cannot exclude the case where $\Sa(\mu)$ is conjugate to $\mathcal{O}_2^+$, as explained in Remark \ref{diag-action}.

\begin{cor}\label{os2'}
Let $\X=\{X(t)\}_{t\ge 0}$ be a proper operator self-similar process in $\R^2$ with an exponent $B$  given in the Jordan form \eqref{Bi}. Then the statements (i)--(iii)
of Theorem \ref{os2} hold verbatim after replacing $\Sa(\mu)$ by $\mathcal{S}(\X)$ and including $\mathcal{O}_2^+$ as a possible symmetry group in (ii)
\end{cor}

\begin{remark}\label{fcp}
As a simple extension of the construction in Remark \ref{OSexists}, we can obtain an operator self-similar process in $\R^2$ with any exponent, and any admissible symmetry group.   Take $X(t)$ as in Remark \ref{OSexists} and let $Y(t)=X(T(t))$ where $T(t)$ is a self-similar process (time change) with $T(at)=a^p T(t)$ (e.g., take $T(t)=t^p$).  Then $Y(t)$ is operator self-similar with exponent $D=pB$.  This, together with Example \ref{O+}, also shows that $\mathcal{S}(\X)$ can take every possible form listed in Corollary \ref{os2'}, which therefore provides a complete characterization in $\R^2$ of the possible symmetries of an o.s.s.~process.  An interesting and useful example of a self-similar process $T(t)$ with Hurst index $0<\beta<1$, which is not infinitely divisible or even Markovian, is given by the first passage or hitting time $T(t)=\inf\{u>0: D(u)>t\}$ of a stable subordinator $D(t)$ with $E(e^{-sD(t)})=\exp(-cts^\beta)$.  The process $Y(t)=X(T(t))$ has densities $h(x,t)$ that solve the space-time fractional multiscaling diffusion equation
\[\frac{\partial^\beta h(x,t)}{\partial t^\beta}=Lh(x,t)\]
where $L$ is the generator of the operator stable semigroup, see for example \cite{gADE,Zsolution,RW2d}.  This fractional diffusion equation models contaminant transport in heterogeneous porous media, and the process $Y(t)$ represents the path of a randomly selected contaminant particle.  The order of the time fractional derivative $\beta$ controls particle retention (sticking or trapping) while the exponent of the operator stable process codes the anomalous superdiffusion caused by long particle jumps.
Also, the inverse process $T(t)$ is constant on intervals corresponding to jumps of the stable subordinator $D(t)$, the length of which is determined by the stable index $\beta$.  Note that the time change need not be independent of the outer process \cite{coupleCTRW,triCTRW}.  Methods for simulating these non-Markovian subordinated processes have recently been developed by Magdziarz and Weron \cite{Magdziarz2007a} and Zhang et al.~\cite{time-Langevin}.
\end{remark}

\begin{remark}
Any proper operator self-similar process can be transformed to one in which the exponent takes the Jordan form \eqref{Bi} by a simple change of basis.  Corollary \ref{os2'} shows that the Jordan basis renders the symmetries orthogonal, and then the Jordan form determines which symmetries are possible.  
\end{remark}


\begin{thebibliography}{20}



\bibitem{Rosinski01a} S{\o}ren Asmussen and Jan Rosi\'nski.
\newblock Approximations of small jumps of L\'evy processes with
a view towards simulation. \newblock {\em J. Appl. Probab.}, 38(2):482--493,
2001.

\bibitem{Barker07} William Barker and Roger Howe. \newblock {\em Continuous Symmetry: From Euclid to Klein}. \newblock American Mathematical Society, 2007.

    \bibitem{coupleCTRW} P. Becker-Kern, M.M. Meerschaert and H.P.
Scheffler (2004) Limit theorems for coupled continuous time random walks.  {\it The Annals of Probability} {\bf 32}, No. 1B, 730--756.

\bibitem{WRRapply} D. Benson, S. Wheatcraft and M. Meerschaert (2000) Application of a fractional advection-dispersion equation. {\it Water Resour. Res.} {\bf 36}, 1403--1412.

\bibitem{BensonTiPM} D. Benson, R. Schumer, M. Meerschaert and S. Wheatcraft (2001) Fractional dispersion, L\'evy motions, and the MADE tracer tests. {\it Transport in Porous Media} {\bf 42}, 211--240.

\bibitem{Billingsley1966} P.~Billingsley (1966) Convergence of types in k-space. {\it Z. Wahrsch. Verw. Geb.} {\bf 5}, 175--179.

\bibitem{BG62} R.~M.~Blumenthal and R.~K.~Getoor,
The dimension of the set of zeros and the graph of a symmetric stable process.
Illinois J. Math. 6 1962 308--316.

\bibitem{Cohen05} S.~Cohen, C.~Lacaux, and M.~Ledoux. (2008)
A general framework for simulation of fractional fields. {\it Stochastic Process. Appl.} {\bf 118}(9), 1489--1517.

\bibitem{Cohen07} S.~Cohen and Rosi\'nski. \newblock Gaussian
approximation of multivariate L\'evy processes with applications
to simulation of tempered stable processes. \newblock {\em Bernoulli},
13(1):195--210, 2007.

\bibitem{Cont04} Rama Cont and Peter Tankov. \newblock {\em Financial
modelling with jump processes}. \newblock Chapman \& Hall/CRC, Boca
Raton, Florida, 2004.

\bibitem{em:mae} P. Embrechts and M. Maejima (2002) {\it
Self-similar Processes}. Princeton University Press.

\bibitem{Graybill83} Franklin~A. Graybill. \newblock {\em Matrices
with applications in statistics}. \newblock Wadsworth Statistics/Probability
Series. Wadsworth Advanced Books and Software, Belmont, Calif., second
edition, 1983.

\bibitem{Hudson82} William~N. Hudson and J.~David Mason. \newblock
Operator-self-similar processes in a finite-dimensional space. \newblock
{\em Trans. Amer. Math. Soc.}, 273(1):281--297, 1982.

\bibitem{hurst} H.E. Hurst, R.P. Black, and Y.M. Simaika (1965) {\it Long-term Storage: An Experimental
Study}, Constable, London.

\bibitem{janicki94} Aleksander Janicki and Aleksander Weron. \newblock
{\em Simulation and Chaotic Behavior of {$\alpha$}-Stable Stochastic
Processes}. \newblock Monographs and Textbooks in Pure and Applied
Mathematics. Marcel Dekker Inc. New York, 1994.

\bibitem{Jurek93} Zbigniew~J. Jurek and J.~David Mason. \newblock
{\em Operator-Limit Distributions in Probability Theory}. \newblock
Wiley Series in Probability and Mathematical Statistics. John Wiley \& Sons Inc., New York, 1993.


\bibitem{Kallenberg02} Olav Kallenberg. \newblock {\em Foundations
of Modern Probability}. \newblock Probability and its Applications
(New York). Springer-Verlag, New York, second edition, 2002.

\bibitem{Kloeden92} Peter~E. Kloeden and Eckhard Platen. \newblock
{\em Numerical solution of stochastic differential equations}.
\newblock Applications of Mathematics. Springer-Verlag, New York,
1992.

\bibitem{Lacaux04} C\'eline Lacaux. \newblock Series representation
and simulation of multifractional L\'evy motions. \newblock {\em
Adv. in Appl. Probab.}, 36(1):171--197, 2004.

\bibitem{Magdziarz2007a} Magdziarz, M., and A. Weron, Competition between subdiffusion and L\'evy flights: A Monte Carlo approach, {\it Phys. Rev. E, 75}, 056702, 2007.

\bibitem{MM94} Maejima, M. and J. D. Mason (1994) Operator-self-similar stable proc\-esses. {\it Stoch. Proc. Appl.} {\bf 54}, 139--163.

\bibitem{Maejima95} Makoto Maejima. \newblock Operator-stable processes
and operator fractional stable motions. \newblock {\em Probab.
Math. Statist.}, 15:449--460, 1995. \newblock Dedicated to the memory
of Jerzy Neyman.

\bibitem{Maejima98} Makoto Maejima. \newblock Norming operators
for operator-self-similar processes. \newblock {\em Trends Math.},
Birkh{ä}user, Boston 287--295, 1998.

\bibitem{marginal} M.M. Meerschaert and H.P. Scheffler, One dimensional marginals of operator stable laws and their domains of attraction, {\it Publ. Math. Debrecen}, {\bf 55}(3--4), 487--499, 1999.

\bibitem{Meerschaert95} Mark~M. Meerschaert and J. A. Veeh.
\newblock Symmetry groups in $d$-space. \newblock {\em Statistics
\& Probability Letters} 22:1--6, 1995.

\bibitem{MAcorr} Mark~M. Meerschaert and Hans-Peter Scheffler. Sample cross-correlations for moving averages with regularly varying tails. {\em Journal of Time Series Analysis},  22(4):481--492, 2001.

\bibitem{Meerschaert01} Mark~M. Meerschaert and Hans-Peter Scheffler.
\newblock {\em Limit Distributions for Sums of Independent Random
Vectors:  Heavy Tails in Theory and Practice}. \newblock Wiley Series in Probability and Statistics. John Wiley \& Sons Inc., New York, 2001.


\bibitem{gADE} Meerschaert, M., D. Benson and B. Baeumer (2001) Operator L\'evy motion and multiscaling anomalous diffusion. {\it Phys. Rev. E} {\bf 63}, 1112--1117.

\bibitem{Zsolution} M.M. Meerschaert, D.A. Benson, H.P. Scheffler and B. Baeumer (2002)  Stochastic solution of space-time fractional diffusion equations. {\it Phys. Rev. E} {\bf 65}, 1103--1106.

\bibitem{portfolio} M.M. Meerschaert and H.P. Scheffler, Portfolio modeling with heavy tailed random vectors. {\em Handbook of Heavy-Tailed Distributions in Finance}, 595--640, S. T. Rachev, Ed., Elsevier North-Holland, New York, 2003.

\bibitem{MarkXiao} Mark~M. Meerschaert and  Yimin Xiao. Dimension results for sample paths of operator stable Levy processes. {\em Stochastic Processes and Their Applications}, 115(1):55--75, 2005.

\bibitem{coupleEcon} M.M. Meerschaert, E. Scalas, Coupled continuous time random walks in finance. {\em Physica A: Statistical Mechanics and Its Applications}, {\bf 370}, 114--118, 2006.

    \bibitem{triCTRW}
Meerschaert, M.M. and H.-P. Scheffler (2008)
\newblock Triangular array limits for continuous time random walks.
\newblock {\em Stoch. Proc. Appl.}, 118, 1606--1633.

\bibitem{PT69} W.E. Pruitt and S.J. Taylor (1969) Sample path properties of processes with stable components.  {\it Z.
Wahrsch. verw. Geb.} {\bf 12}, 267--289.


\bibitem{MittnikR} Rachev, S. and S. Mittnik (2000) {\it Stable Paretian Models in Finance}, Wiley, Chichester.

\bibitem{Rosinski89} B.S. Rajput and J. Rosi\'nski. \newblock Spectral
representations of infinitely divisible processes. \newblock {\em
Probab. Th. Rel. Fields}, 82: 451--487, 1989.

\bibitem{Reeves2} D.M. Reeves, D.A. Benson, M.M. Meerschaert, H.P. Scheffler, Transport of Conservative Solutes in Simulated Fracture Networks 2. Ensemble Solute Transport and the Correspondence to Operator-Stable Limit Distributions, {\it Water Resources Research}, 44 (2008), W05410.

\bibitem{Rosinski90} J. Rosi\'nski. \newblock On series representations
of infinitely divisible random vectors. \newblock {\em Ann. Probab.},
82: 405--430, 1990.

\bibitem{Rosinski01b} Jan Rosi\'nski. \newblock Series representations
of L\'evy processes from the perspective of point processes. \newblock
In {\em L\'evy processes}, pages 401--415. Birkhäuser Boston, Boston,
MA, 2001.

\bibitem{Rosinski04} Jan Rosi\'nski. \newblock Tempering stable
processes. \newblock{\em Bernoulli}, 13:195-210, 2007.

\bibitem{Sato} Ken-Iti Sato.
Strictly operator-stable distributions. {\it Journal of Multivariate Analysis}, 22:278--285, 1987.


\bibitem{Sato99} Ken-iti Sato. \newblock {\em L\'evy Processes and
Infinitely Divisible Distributions}, volume~68 of {\em Cambridge
Studies in Advanced Mathematics}. \newblock Cambridge University
Press, Cambridge, 1999. \newblock Translated from the 1990 Japanese
original, Revised by the author.

\bibitem{SGMM} E. Scalas, R. Gorenflo, F. Mainardi, and M.M. Meerschaert, Speculative option valuation and the fractional diffusion equation, {\it Fractional Derivatives and Their Applications}, 265--274, A. Le Mehaut\'e, J. A. Tenreiro Machado, J. C. Trigeassou and J. Sabatier, Eds. (2005), Ubooks, Germany.

\bibitem{Sheluhin} Oleg Sheluhin, Sergey Smolskiy, Andrew Osin, {\it Self-Similar Processes in Telecommunications}, Wiley, New York, 2007.

\bibitem{RW2d} Y. Zhang, D.A. Benson, M.M. Meerschaert, E. M. LaBolle, and H.P. Scheffler. Random walk approximation of fractional-order multiscaling anomalous diffusion. {\em Physical Review E}, 74(2):6706--6715, 2006.


\bibitem{time-Langevin} Y. Zhang, M.M. Meerschaert, B. Baeumer (2008) Particle tracking for time-fractional diffusion, {\it Physical Review E}, {\bf 78}(3), 036705.

\end{thebibliography}
\end{document}